\documentclass[12pt]{article}
\usepackage{amssymb}
\usepackage{float}
\usepackage{multirow}

\usepackage{amsmath}
\numberwithin{equation}{section}
\usepackage{booktabs}
\usepackage{float}
\usepackage{amsthm}
\usepackage{xcolor}
\usepackage{arydshln}
\newtheorem{thm}{Theorem}[section]
\newtheorem{lem}{Lemma}[section]
\newtheorem{prop}{Proposition}[section]
\newtheorem{remark}{Remark}[section]
\usepackage{amssymb,amsmath}
\usepackage{a4wide}
\usepackage{graphicx}
\usepackage{setspace}
\usepackage{natbib}
\usepackage{multicol}
\def\i{\mathrm{i}}

\usepackage{enumerate}
\newcommand{\Cov}{{\rm Cov}}
\newcommand{\Var}{{\rm Var}}
\title{From nonstationarity to stationarity via $1/f$ noise: discrete Fourier transforms and sample mean asymptotics for testing}
\author{ Mohamedou Ould Haye$^1$\footnote{corresponding author : {mohamedouhaye@cunet.carleton.ca}} \and
  Anne Philippe$^2$
}
\date{$^1$\small  School of Mathematics and Statistics. \\
Carleton University, 1125 Colonel By Dr. Ottawa, ON, Canada, K1S 5B6
\\ $^2$ Nantes Université, CNRS, Laboratoire de Mathématiques Jean Leray, LMJL,
UMR 6629, F-44000 Nantes, France}
\begin{document}
\maketitle
\begin{abstract}
We study the asymptotic behaviour of different statistics for time series exhibiting long memory and nonstationarity. For processes with memory parameter $d\in(-1/2,3/2)$, we derive the joint limiting distribution of  discrete Fourier transforms at a fixed number of Fourier frequencies, with a unified normalization. The resulting limits are Gaussian with an explicit covariance structure.
Particular attention is given to the boundary case $d=1/2$, also known as $1/f$ noise. We show that logarithmic corrections yield nondegenerate limits for sample mean and sample variance leading to explicit asymptotic distributions of $\chi^2$ type. We construct a statistic that combines the sample mean, the sample variance, and low-frequency periodogram ordinates, designed so that, at the boundary case $(d=1/2)$, it admits a tractable limit distribution.
 These results are applied to construct a consistent  parameter-free  test of nonstationarity against long memory stationarity.
\end{abstract}
\textbf{Keywords:} Long memory, dependence, stationarity and nonstationarity, spectral methods, discrete Fourier transform, periodogram, nonparametric testing.
\section{Introduction}
An extension of the usual class of long- and short-memory processes, for which the memory parameter $d$ lies in $(-1/2, 1/2)$, to the broader range $d \in (-1/2, 3/2)$ has been introduced by \cite{Ray}. In this framework, processes with $d \in (-1/2, 1/2)$ correspond to the stationary case, and those with $d \in [1/2, 3/2)$ to the nonstationary case. In the latter case, the increments belong to the stationary framework with memory parameter $d-1 \in [-1/2, 1/2)$. An extensive literature is available for the range $d \in (-1/2, 1/2)$, both in terms of asymptotic theory and inferential methods; see, for example, \cite{MR2977317} for a review.
Statistical inference for the nonstationary processes has been studied by several authors. For instance, \cite{VELASCO1999325} extend the semiparametric estimator of \cite{Robinson1995}, while further contributions include \cite{Shao2010} and \cite{Abadir2007}, among many others.
Testing stationarity has also been the focus of several studies. In particular, distinguishing long-memory dependence from genuine nonstationarity remains a fundamental challenge in time series analysis, as both phenomena lead to slowly decaying sample autocovariances and similar low-frequency behaviour.
For fully parametric models, \cite{Robinson,Robinson2003} addresses the problem of testing nonstationarity against stationarity in the presence of long memory; his approach is described in Section~\ref{Robinson method}.\\
To test stationarity in presence of long memory against nonstationarity, \cite{MR2328526} propose a time-domain approach ($V/S$ test)  based on the asymptotic behaviour of partial sums and the estimation of the memory parameter $d$. However, the  V/S empirical size becomes very high compared to the nominal level as $d$ approaches the  boundary $1/2$, where the distinction between stationary and nonstationary behaviour becomes particularly subtle.

In this paper, we provide a nonparametric testing procedure of nonstationarity.
 Specifically, we consider the hypothesis testing, 
\[
\begin{aligned}
H_0:\quad &
\left\{
\begin{array}{l}
X_t\textrm{ is nonstationary } \; \tfrac12 \le d < \tfrac32, \\
\text{or} \\
X_t=g_n(t)+Y_t,\,\, Y_t\ \textrm{is zero-mean stationary with } -\tfrac12 < d < \tfrac12,\\
\textrm{with }g_n(t)\textrm{ being a nonparametric trend,}
\end{array}
\right.
\\[1em]
H_1:\quad & X_t\textrm{ is stationary }\;\textrm{with } -\tfrac12 < d < \tfrac12.
\end{aligned}
\]
The null hypothesis therefore encompasses both stochastic and deterministic sources of nonstationarity. The former corresponds to fractional integration with $d\ge 1/2$, while the latter allows for deterministic components such as structural breaks in the mean. The deterministic component $g_n(t)$ will be specified later.\\
The boundary value $d=1/2$ marks the transition between stationary and nonstationary regimes and corresponds to the so--called $1/f$ noise model, referred to by \cite{Robinson} as the ``just nonstationary'' case.\\
We construct a test statistic ($Q_n(s)$) from the sample mean, the sample variance, and a finite number of low Fourier frequencies, and discriminates between the nonstationary and stationary regimes without requiring estimation of the memory parameter $d$. The resulting procedure provides a simple and computationally efficient method for distinguishing stationary behaviour from fractional or deterministic sources of nonstationarity. In particular, the procedure does not rely on parametric assumptions on the linear filter, as in \cite{Robinson}. 

The key theoretical result, Theorem~\ref{finite dim}, establishes the asymptotic behaviour of the discrete Fourier transform (DFT)  of the observed series at different frequencies under both stationary and nonstationary regimes.
Note that when \(d\in(-1/2,1/2)\), under assumptions on the spectral density, \cite{MR1307342} derived the asymptotic distribution of the periodogram at low frequency ordinates, without going through DFT asymptotic behaviour.\\
 The limiting distribution of $Q_n(s)$ can be expressed as a finite weighted sum of independent $\chi^2(1)$ random variables, where the weights are the eigenvalues of a positive--definite covariance matrix. This tractable representation contrasts with the stochastic--integral limits typically arising in classical time--domain tests (for instance, \cite{DickeyFuller1979} and  \cite{MR2328526}).
Beyond its role in constructing the test statistic $Q_n(s)$, Theorem~\ref{finite dim} is of independent interest, as it characterizes the asymptotic behaviour of the DFT for both stationary and integrated processes within a unified framework.
Theorem~\ref{1 over f1} and Proposition~\ref{unit root1} provide insight into the behaviour of partial sums at the stationarity--nonstationarity boundary $d=1/2$, showing that a normalization factor of $\log(n)$ is required for the sample mean and sample variance to admit nondegenerate limits in this case. This behaviour of the sample mean motivates its inclusion in the test statistic, rather than building it solely from periodograms at Fourier frequencies.
More generally, Theorem~\ref{1 over f1} provides the asymptotic distribution over the entire nonstationary range $d \ge 1/2$, thereby covering the full class of nonstationary processes with $1/2\le d<3/2$.
Simulation experiments illustrate that the proposed statistic maintains empirical rejection rates close to the nominal level at the boundary case \(d=1/2\). The experiments also show that the method remains stable in the presence of deterministic components such as structural breaks.
The remainder of the paper is organized as follows. Section \ref{sec:2} presents the main theoretical results. Section \ref{simulation} describes the testing procedure. Section \ref{robinson sim} reports simulation results illustrating the performance of the proposed statistic ($Q_n(s)$) alongside the nonstationarity tests proposed in \cite{Robinson}.
Section \ref{proofs} contains the proofs.
\section{Main Results}\label{sec:2}
The class of stationary processes that will be considered here corresponds to linear processes \((X_t)_{t \ge 0}\)  with memory parameter \(d \in [-\tfrac{1}{2}, \tfrac{1}{2})\)   
\begin{equation}\label{linear}
X_t =\mu+\sum_{j=0}^\infty a_j\, \epsilon_{t-j},
\end{equation}
where $\mu\in\mathbb{R}$, \((\epsilon_t)\) are independent and identically distributed random variables with zero mean, variance \(\sigma_\epsilon^2 = \mathbb{E}(\epsilon_1^2)\), and finite fourth moment. 
The coefficients \((a_j)\) determine the degree of memory through their asymptotic behavior:
\begin{equation}\label{summable}
\sum_{j=0}^\infty |a_j| < \infty, \quad \sum_{j=0}^\infty a_j\neq0,\qquad\text{when } d = 0,
\end{equation}
\begin{equation}\label{positived}
a_j \sim c(d)\, j^{-1+d}, \quad \text{for } 0 < d < \tfrac{1}{2},
\end{equation}
\begin{equation}\label{negatived1}
\sum_{j=0}^\infty a_j = 0, \qquad 
a_j = c(d)\, j^{-1+d}\,(1 + O(j^{-1})), \quad \text{for } -\tfrac{1}{2} \le d < 0,
\end{equation}
where \(c(d) \neq 0\).  \\\\
    Condition (\ref{summable}), imposed when $d=0$, implies that the corresponding covariance function $\gamma$ of $X_t$ satisfies
 $$
 \sum_{h=-\infty}^\infty\vert\gamma(h)\vert<\infty.
 $$
 Also, according to Proposition 3.2.1. of \cite{MR2977317}, when $d\neq 0$ $(-1/2<d<1/2)$  the conditions \eqref{positived} and   (\ref{negatived1}) imply that
 \begin{equation}\label{gammad}
\gamma(h)\sim C(d) h^{2d-1},\qquad\textrm{as }h\to\infty,\qquad C(d)=\sigma^2_\epsilon c^2(d)B(d,1-2d),
 \end{equation}
 where $B$ is the $\beta$ function (extended to the interval $(-1,0)$), and 
 $a_n\sim b_n$ means that $a_n/b_n\to1$ as $n\to\infty$.\\\\
We say that nonstationary process \((X_t)\) is an integrated process with memory parameter \(1/2 \le d < 3/2\), if its first-differenced process
\[
Y_t = X_t - X_{t-1}
\]
is stationary linear  process with memory parameter \(d - 1\) satisfying \eqref{linear} with $\mu=0$, and  one of \eqref{summable}, \eqref{positived} or \eqref{negatived1}.
Let
$\Sigma(d)$ be the block matrix
\begin{equation}\label{Sigma}
\Sigma(d) =
\begin{pmatrix}
\Sigma^{(c)}(d) & 0\\[3pt]
0 &\Sigma^{(s)}(d)
\end{pmatrix},
\end{equation}
where for 
 $i,j=1,\ldots,s$, 
\[
\Sigma^{(c)}_{ij}(d) =
\begin{cases}
\displaystyle \frac{\delta(-1/2)}{2}\int_{[0,1]^2}\!\cos(2\pi i x)\cos(2\pi j y)\,
\bigl(-\log|x-y|\bigr)\,dx\,dy,
& \text{if } d=\tfrac{1}{2},\\[10pt]
\displaystyle\frac{\delta(d-1)}{2}\int_{[0,1]^2}\!\cos(2\pi i x)\cos(2\pi j y)\,
|x-y|^{2d-1}\,dx\,dy,
& \text{if } 1/2<d<3/2,\\[10pt]
-\delta(d)\Bigg[a_{ij}(d)+
2\pi^2 i j\displaystyle\int_{[0,1]^2}\sin(2\pi ix)\sin(2\pi jy)\vert x- y\vert^{2d+1}dxdy\Bigg], 
& \text{if } -1/2<d<1/2,
\end{cases}
\]
where
$$
a_{ij}(d)=1-(2d+1)\int_0^1x^{2d}\left(\cos(2\pi ix)+\cos(2\pi jx)\right)dx,
$$
and
\begin{equation}\label{const delta}
\delta(d)=\begin{cases}8\sigma^2_\epsilon c^2(-1/2),&\textrm{if }d=-1/2,\\\\
\sigma^2_\epsilon\left(\displaystyle\sum_{i=0}^\infty a_i\right)^2,&\textrm{if }d=0,\\\\
   \frac{C(d)}{d(2d+1)},&\textrm{if }0<\vert d\vert<1/2.
\end{cases}
\end{equation}
Similarly, we define
\[
\Sigma^{(s)}_{ij}(d) =
\begin{cases}
\displaystyle\frac{\delta(-1/2)}{2} \int_{[0,1]^2}\!\sin(2\pi i x)\sin(2\pi j y)\,
\bigl(-\log|x-y|\bigr)\,dx\,dy,
& \text{if } d=\tfrac{1}{2},\\[10pt]
\frac{\delta(d-1)}{2}\displaystyle\int_{[0,1]^2}\!\sin(2\pi i x)\sin(2\pi j y)\,
|x-y|^{2d-1}\,dx\,dy,
& \text{if } 1/2<d<3/2,\\[10pt]
-\delta(d)(2\pi^2 ij)\displaystyle\int_{[0,1]^2}\!\cos(2\pi i x)\cos(2\pi j y)\,
|x-y|^{2d+1}\,dx\,dy,
& \text{if } -1/2<d<1/2.
\end{cases}
\]
The following lemma will be used in the proof of Theorem \ref{finite dim}.
\begin{lem}\label{invertible}
The matrices $\Sigma^{(c)}(d)$ and $\Sigma^{(s)}(d)$ are positive definite, and hence $\Sigma(d)$ is positive definite as well. 
\end{lem}
In the next theorem \ref{finite dim}  we give the  asymptotic
joint distribution of the real and imaginary parts of the discrete Fourier transform (DFT) for stochastic processes  with memory parameter $-1/2<d<3/2$. This multivariate convergence is central to establish the subsequent results in Section 3. Theorem \ref{1 over f1}  provides asymptotic distribution of the sample mean and gives $\log(n)$ correction that is needed to obtain a nondegenerate asymptotic distribution in case of border $d=1/2$.
\begin{prop}
     \label{unit root1} (i) If $Y_t$ is a covariance-stationary process having a spectral density of the form
$$
f(\lambda)\sim c_f\vert\lambda\vert,\qquad\textrm{as }\lambda\to0,\qquad c_f>0,
$$
then
$$
\textrm{Var}(Y_1+\cdots+Y_n)\sim 4c_f\log(n),\qquad\textrm{ as }n\to\infty.
$$
(ii) If $Y_t$ is a linear process of the form \eqref{linear} with $d=-1/2$
   then the  spectral density of $Y_t$ satisfies
$$
f(\lambda)\sim c_f\vert\lambda\vert,\qquad c_f=2\sigma^2_\epsilon c(-1/2)^2,
$$
and therefore
\begin{equation}\label{var sum 1/2}
    \textrm{Var}(Y_1+\cdots+Y_n)\sim\delta(-1/2)\log(n),\qquad\textrm{as }n\to\infty.
\end{equation}
\end{prop}
\begin{thm}
\label{finite dim} 
Let $(X_t)$ be a stochastic process, integrated or stationary, with memory parameter $d \in (-1/2,\,3/2)$. Then for  Fourier frequencies  $\lambda_1,\ldots,\lambda_s$, with fixed $s\ge1$, we have 
\begin{eqnarray*}\lefteqn{
Z_n(s,d):=\Bigg[\left(\frac{1}{n^{1/2+d}}\sum_{k=1}^n\cos(k\lambda_{1})X_k\right),
\cdots,\left(\frac{1}{n^{1/2+d}}\sum_{k=1}^n\cos(k\lambda_{s})X_k\right),}\nonumber\\
&& \left(\frac{1}{n^{1/2+d}}\sum_{k=1}^n\sin(k\lambda_1)X_k\right),\cdots,\left(\frac{1}{n^{1/2+d}}\sum_{k=1}^n\sin(k\lambda_{s})X_k\right)\Bigg]\overset{d}{\rightarrow}\mathcal{N}\left(0,\Sigma(d)\right), 
\end{eqnarray*}
where $\Sigma(d)$ is defined in \eqref{Sigma}.
\end{thm}
The following theorem provides the asymptotic distribution of the sample mean for integrated processes. For  linear stationary processes, its behaviour is well understood (see, for example, \cite{ibragimov_linnik_1971}, Theorem~18.6.5).
\begin{thm}\label{1 over f1}
Let $(X_t)_{t\ge1}$ be an integrated process with memory parameter $\tfrac12\le d<3/2$.\\\\
{\bf Part A}
Let
$$
d_n=\begin{cases}
    \sqrt{\log(n)},&\textrm{if }d=\frac{1}{2},\\\\n^{d-\frac{1}{2}},&\textrm{if }\frac{1}{2}<d<\frac{3}{2}.
\end{cases}
$$
Then
\[
\frac{\bar X_n}{d_n} \;\xrightarrow{d}\; \mathcal N\!\left(0,\tfrac1{2d+1}\delta(d-1)\right).
\]
{\bf Part B} Assume $d=1/2.$ 
\begin{enumerate}
\item
\[
\frac{1}{n\log n}\sum_{k=1}^n (X_k-\bar X_n)^2
\;\xrightarrow{L^2}\;
\frac12\delta(-1/2),
\]
and therefore
\[
\frac{1}{n\log n}\sum_{k=1}^n X_k^2
\;\xrightarrow{d}\;
\frac12\delta(-1/2)\bigl(1+\chi^2(1)\bigr),
\]
and 
$$
\frac{\bar X_n^2}{\frac{1}{n}\displaystyle\sum_{k=1}^n (X_k-\bar X_n)^2}
\;\xrightarrow{d}\;
\chi^2(1).
$$
\item Let $m=[\,\sqrt{n}\,]$ and
\[
D_n=\frac{1}{m}\Big(\sum_{i=1}^m X_i-\sum_{i=n-m+1}^n X_i\Big).
\]
Then
\[
\frac{D_n}{\sqrt{\log n}} \;\xrightarrow{d}\; \mathcal N\!\left(0,\tfrac12\delta(-1/2)\right).
\]
and therefore
\begin{equation}\label{sample mean ratio}
\frac{D_n^2}{\frac{1}{n}\displaystyle\sum_{k=1}^n (X_k-\bar X_n)^2}
\;\xrightarrow{d}\;
\chi^2(1).
\end{equation}
\item $\bar X_n/\sqrt{\log(n)}$ and $D_n/\sqrt{\log(n)}$ are asymptotically independent of $Z_n(s,1/2)$ defined in Theorem \ref{finite dim}.
\end{enumerate}
\end{thm}
\section{Testing Procedure for  nonstationarity against stationarity }\label{simulation}

Testing nonstationarity against stationarity is formulated as follows:
\[
\begin{aligned}
H_0:\quad &
\left\{
\begin{array}{l}
X_t\textrm{ is nonstationary } \; \tfrac12 \le d < \tfrac32, \\
\text{or} \\
X_t=g_n(t)+Y_t,\,\, Y_t\ \textrm{is zero-mean stationary with } -\tfrac12 < d < \tfrac12,\\
\textrm{with }g_n(t)\textrm{ being a nonparametric trend,}
\end{array}
\right.
\\[1em]
H_1:\quad & X_t\textrm{ is stationary }\;\textrm{with } -\tfrac12 < d < \tfrac12.
\end{aligned}
\]
Thus the null hypothesis allows for both stochastic nonstationarity generated by fractional integration with $d \ge 1/2$, and deterministic nonstationarity arising from structural breaks or trends represented by $g_n(t)$. The alternative corresponds to stationarity, that encompasses linear processes with long, persistent or short memory.

The construction of the critical region relies on the following convergence result. This proposition also allows the null hypothesis to be extended to nonstationary processes generated by a deterministic trend with additional stationary noise.
\begin{prop}\label{periodogram variance}
 For a stochastic process $X_t$, $t\ge1$, consider the following statistic,
 $$
 \tilde Q_n(s):=\frac{D_n^2}{\frac{1}{n}{\displaystyle\sum_{k=1}^n(X_k-\bar X_n)^2}}
  +\log (n)\displaystyle\frac{\sum_{j=1}^sI_n(\lambda_j)}{\displaystyle\sum_{k=1}^n(X_k-\bar X_n)^2},
 $$
where $I_n(\cdot)$ is the periodogram  built from $(X_1,...,X_n)$,  $s$ is a fixed positive integer and $ \lambda_j=2\pi j/n, \quad j=1,\ldots,s,$ denote the first Fourier frequencies:
\begin{equation}\label{non standard perio}
I_n(\lambda_j)=\frac{1}{n}\left\vert\sum_{k=1}^nX_ke^{\i k\lambda_j}\right\vert^2.
\end{equation}

\begin{enumerate}
    \item (Stationarity or stochastic trend)
Let $(X_t)$ be an integrated or stationary  stochastic process with memory parameter $d \in (-1/2,\,3/2)$. Then, 
$$
\tilde Q_n(s)\overset{d}{\rightarrow}\begin{cases}
0,&\textrm{if }d<1/2,\\\\
\displaystyle\sum_{i=0}^{2s}\psi_iQ_i,&\textrm{if }d=1/2,\\\\
\infty,&\textrm{if }d>1/2,
\end{cases}
$$
where $Q_i$ are i.i.d $\chi^2(1)$ random variables, $\psi_0=1$, and $\psi_i$, $i\ge1$ are the eigenvalues of the matrix
\begin{equation}\label{Sigma variance}
\Sigma =
\begin{pmatrix}
\Sigma^{(c)} & 0\\[3pt]
0 &\Sigma^{(s)}
\end{pmatrix},
\end{equation}
where for 
 $i,j=1,\ldots,s$, 
\[
\Sigma^{(c)}_{ij} =
\int_{[0,1]^2}\!\cos(2\pi i x)\cos(2\pi j y)\,
\bigl(-\log|x-y|\bigr)\,dx\,dy,
\]
and
$$
\Sigma^{(s)}_{ij} =
\int_{[0,1]^2}\!\sin(2\pi i x)\sin(2\pi j y)\,
\bigl(-\log|x-y|\bigr)\,dx\,dy.
$$
\item Deterministic trend (structural breaks):  Let $X_t=g_n(t)+Y_t$ where  $Y_t$ is a zero-mean linear stationary process with $d\in(-1/2,1/2)$ and $g_n(t)=n^\beta g(t/n)$, with $\beta\ge0$, and where $g$ is a piece-wise continuous\footnote{the interval [0,1] can be divided into a finite number of sub-intervals such that the function is continuous on each open sub-interval, with finite limit at each endpoint.}   function satisfying  
\begin{equation}\label{nonzero}
\int_0^1g(x)e^{\i 2\pi jx}dx\neq0, \qquad\textrm{for at least one }j\in\{1,2,\ldots,s\}.
\end{equation}
Then $\tilde Q_n(s)\overset{P}{\to}\infty.$
\end{enumerate}
\end{prop}
\paragraph{Implementing the test} An immediate consequence of Proposition  \ref {periodogram variance} is to take as rejection region for testing $H_0$ against $H_1$, defined in the beginning of this section 
\[
R=\{\tilde Q_n(s)<q(\alpha,s)\},
\]
where $q(\alpha,s)$ denotes the $\alpha$-quantile of the limiting distribution
\[
\sum_{i=0}^{2s}\psi_i Q_i
\]
introduced in Proposition~\ref{periodogram variance}. This yields a consistent test with asymptotic level $\alpha$. Numerical studies conducted in the next section, aimed at improving the test power, used a slightly modified version 
 of the statistic $\tilde Q_n(s)$:
\[
Q_n(s):=\frac{D_n^2}{\frac{1}{n}{\displaystyle\sum_{k=1}^n(X_k-\bar X_n)^2}}
+\frac{\displaystyle\sum_{j=1}^s I_n(\lambda_j)}
{\displaystyle\sum_{k=1}^n\frac{(X_k-\bar X_n)^2}{\log(k+1)}}.
\]
Proposition~\ref{periodogram variance} remains valid with $\tilde Q_n(s)$ replaced by $Q_n(s)$. The argument is essentially unchanged, since it only requires showing that
\[
\frac{1}{n}\sum_{k=1}^n\frac{(X_k-\bar X)^2}{\log(k+1)}-
\frac{1}{n\log(n)}\sum_{k=1}^n (X_k-\bar X)^2\overset{L^2}{\longrightarrow}0,
\]
which is a consequence of summation by parts and Cesaro $L^2$ convergence.

We therefore consider the rejection region 
\[
R=\{Q_n(s)<q(\alpha,s)\}.
\]
\begin{remark}
    The statistic $Q_n(s)$ can also be used to run $1/f$ two side test, i.e., test $d=1/2$ against $d\neq 1/2$. The rejection region of level $\alpha$ would then be 
    \[
\{Q_n(s)<q(\alpha/2,s)\}\cup\{Q_n(s)>q(1-\alpha/2,s)\}.
\]
\end{remark}
\section{Simulation Results}\label{robinson sim}
\subsection{Selection of the Tuning Parameter $s$.}
Table \ref{tab:sum_compare} reports the rejection rates (i.e., the empirical size and power) of our statistic \(Q_n(s)\) for various values of \(s\), the number of frequencies used in constructing the statistic. The results are presented for two markedly different sample sizes, \(n=500\) and \(n=2000\).
We aimed at choosing $s$ that maximizes power without degrading the empirical size.
Overall, the results suggest that relatively large values of \(s\) lead to better performance in terms of power. In particular, for moderately large samples, taking (s=10) appears to provide a good tradeoff between power and size.
while \(s=25\) seems more appropriate for larger samples. These choices provide a reasonable balance between empirical size and power across the range of designs considered.
It is also worth noting that the rejection rates are affected by the presence of a short-memory component in the data generating process (DGP). In general, short-memory dynamics tend to reduce the rejection rates, leading to a decrease in both empirical size and power relative to the pure fractional noise case.
\begin{table}[H]
\centering
\caption{Rejection rates for the statistic $Q_n(s)$, using 2000 Monte-Carlo replications. For $d=0.5$, the bold entry is closest to $0.05$. For $d<0.5$, the bold entry is closest to $1$. DGP is FARIMA(0,d,0) with $\mathcal{N}(0,1)$ innovations.}
\label{tab:sum_compare}
\begin{tabular}{cccccccc}
\hline
$n$ & $d$ & $s=1$ & $s=5$ & $s=10$ & $s=15$ & $s=20$ & $s=25$ \\
\hline
\multirow{6}{*}{500} & 0.5 & \textbf{0.058} & 0.070 & 0.078 & 0.085 & 0.088 & 0.092 \\
 & 0.4 & 0.118 & 0.255 & 0.335 & 0.374 & 0.402 & \textbf{0.421} \\
 & 0.3 & 0.276 & 0.630 & 0.719 & 0.755 & 0.770 & \textbf{0.777} \\
 & 0.2 & 0.483 & 0.850 & 0.910 & 0.934 & 0.943 & \textbf{0.948} \\
 & 0.1 & 0.693 & 0.967 & 0.988 & 0.992 & 0.995 & \textbf{0.996} \\
 & 0.0 & 0.845 & 0.997 & \textbf{1.000} & 1.000 & 1.000 & 1.000 \\
\hline
\multirow{6}{*}{2000} & 0.5 & \textbf{0.050} & 0.058 & 0.060 & 0.061 & 0.066 & 0.065 \\
 & 0.4 & 0.146 & 0.339 & 0.441 & 0.487 & 0.523 & \textbf{0.540} \\
 & 0.3 & 0.352 & 0.750 & 0.834 & 0.867 & 0.885 & \textbf{0.895} \\
 & 0.2 & 0.625 & 0.941 & 0.978 & 0.988 & 0.991 & \textbf{0.993} \\
 & 0.1 & 0.845 & 0.996 & 0.999 & 1.000 & \textbf{1.000} & 1.000 \\
 & 0.0 & 0.956 & \textbf{1.000} & 1.000 & 1.000 & 1.000 & 1.000 \\
\hline
\end{tabular}
\end{table}

\subsection{Comparison to Robinson Efficient Nonstationarity Test}\label{Robinson method}

We evaluate the performance of the test statistic $Q_n(s)$ for three examples of data-generating processes and examine how it compares with the closest existing procedure.
\cite{Robinson, Robinson2003} considers fractionally integrated models of the form
\[
(1-B)^dX_t=U_t,
\qquad \text{where } U_t \text{ may exhibit short memory},
\]
and proposes a test of the nonstationarity hypothesis \(d\ge 1/2\) against the alternative \(d<1/2\). The test is based on the statistic
\[
\tilde r=\frac{\sqrt{n}}{\tilde\sigma^2\sqrt{\tilde A}}\,\tilde a,
\]
which converges to the standard normal distribution under the boundary case \(d=1/2\) and converges to zero for $d<1/2$. We follow his notation, with some minor simplifications. Here
\[
\tilde a=-\frac{1}{n}\sum_{j=1}^n 
\log\!\left(2\sin\!\left(\frac{\pi j}{n}\right)\right)
I_U(\lambda_j),\quad
\tilde\sigma^2=\frac{1}{n}\sum_{j=1}^nU^2_j,
\quad
\tilde A=\frac{2}{n}\sum_{j=1}^n
\left[
\log\!\left(2\sin\!\left(\frac{\pi j}{n}\right)\right)
\right]^2,
\]
where \(I_U\) denotes the periodogram based on \(U_t\) as defined in \eqref{non standard perio}.

In practice, \(U_t\) is obtained by applying the fractional differencing operator to the observed series,
\[
U_t=(1-B)^\frac{1}{2}X_t,
\]
which involves the infinite expansion of the fractional differencing operator and must therefore be truncated in empirical work. In addition, when \(U_t\) exhibits short memory, Robinson proposes fitting an AR(\(q\)) model to \(U_t\) and constructing the periodogram using the residuals \(\tilde U_t\) from this fit. Thus, unlike our test statistic \(Q_n(s)\), which is computed directly from the observed data \(X_t\), Robinson's procedure requires both fractional filtering and an additional prewhitening step.
\subsubsection{Fractionally integrated processes}
Table~\ref{tab:three_tests_rejection_rates_AR3} compares nonstationarity empirical rejection rates when sampling from FARIMA $(p,d,0)$ models with $p=0$ and $p=1$. To allow additional flexibility for the prewhitened version of $\tilde r$, the prewhitening step selects among low-order AR$(q)$ models with $q\le3$ using the AIC criterion, although the true model contains at most an AR(1) component. The results show that the statistic $Q_n(s)$ outperforms $\tilde r$ and its prewhitened version in the presence of an AR component. The results reveal a clear contrast between the procedures. The statistic $\tilde r$ performs well when the data-generating process coincides with the fully parametric FARIMA$(0,d,0)$ specification for which it is designed. However, its performance deteriorates markedly once an AR component is present. Prewhitening partially alleviates this loss of power in the presence of short-memory dynamics, but it tends to perform poorly when no short-memory component is present. In contrast, although the statistic $Q_n(s)$ is also affected by the presence of short-memory AR components, the deterioration in its performance is considerably less pronounced.
\subsubsection{Aggregated Random Coefficient AR(1)}
In addition to the  FARIMA DGP, we also consider in Table~\ref{tab:three_tests_rejection_rates_arcar} long memory generated through the aggregation of heterogeneous short-memory processes. Specifically, we generate random-coefficient AR(1) processes
\[
Y_{j,t}=\phi_j Y_{j,t-1}+\varepsilon_{j,t}, \qquad j=1,\ldots,M,
\]
and define the aggregated series
\[
X_{t,M}=\frac{1}{\sqrt{M}}\sum_{j=1}^M Y_{j,t}.
\]
The autoregressive coefficients $\phi_j$ are drawn from a distribution on $(0,1)$ whose density $f$ determines the behaviour of the autocovariance function of the limiting process through its behaviour near $1$. Under the assumption
\[
\mathbb{E}\!\left((1-\phi_1^2)^{-1}\right) < \infty,
\]
as $M\to\infty$, the process $X_{t,M}$ converges in distribution to a stationary Gaussian process with autocovariance function
\[
\gamma(h)=\mathbb{E}\!\left(\phi_1^{|h|}(1-\phi_1^2)^{-1}\right).
\]
See \cite{Leipus2014Aggregation,MR3075595,MR2977317} for more results on these models. 
In the simulation, we take $\phi_1$ with a density of the form
\begin{equation} \label{mixture}
     f(x) = \dfrac{2}{\beta (a,b)} (1-x^2)^{b-1 } x^{2a -1},
\end{equation}
with $a>0$ and $b>0$. Then the limiting aggregated process $X_t$ exists if $b>1$, and for $h\geq 0$
\[
\gamma(h) =
\frac{\Gamma (a+b)}{(b-1) \Gamma(a)}
\frac{\Gamma (a+h/2)}{ \Gamma(a+b-1 +h/2)}
\sim
\frac{\Gamma (a+b)}{ 2^{1-b}(b-1) \Gamma(a)} h^{1-b},\qquad\textrm{as }h\to\infty.
\]
Therefore the limiting aggregated process $X_t$ exhibits long-range dependence when $b\in(1,2)$, in the sense that the covariance function is not summable, and the corresponding long-memory parameter is $d = 1-b/2$.\\
Although the aggregated process may mimic the low-frequency behaviour of fractional models, it does not belong to the class of fractional long-memory processes when the coefficients satisfy \eqref{mixture}; see \cite{Celov01052010}. The results reported in Table~\ref{tab:three_tests_rejection_rates_arcar} indicate that the statistic $Q_n(s)$ remains relatively robust in this setting, displaying a low empirical size even though the data-generating mechanism departs from the fractional linear framework.At the same time, the test retains good empirical power. By contrast, the power of $\tilde r$ and its prewhitened version deteriorates substantially for moderately large values of $a$, which is largely expected since these tests are tailored to a specific parametric framework, namely FARIMA($p,d,0$).
\subsubsection{Random coefficient AR(1) with regime switching}
We consider another class of stationary processes  studied in \cite{Leipus2003}:
\begin{equation}\label{car}
X_t = a_tX_{t-1}+\zeta_t
\end{equation}
where \(a_t\) follows a renewal–reward structure:
\[
a_t = A_j, \qquad S_{j-1} < t \le S_j, \qquad j \in \mathbb{N},
\]
where \((S_j)\) is a renewal process, with positive integer-valued increments $\Delta_j$, and \((S_j, A_j)\) is independent of \((\zeta_t)\).
They showed that $X_t$ exhibits long memory if \((\Delta_j)\) and \((A_j)\) are i.i.d.\ satisfying
\begin{equation*}P(\Delta_1=k)\sim c_\Delta\, k^{-\alpha}, \qquad  k \to \infty,
\qquad 3 < \alpha < 4,
\end{equation*}
$$
P(A_1 = 1)=p=1-P(A_1=c)\in(0,1), \qquad 0<c<1,
$$
The process alternates between unit-root regimes, where it behaves locally like a random walk, and stationary regimes corresponding to an AR(1) process with parameter $0<c<1$. However, (\ref{car}) admits a stationary solution, see \cite{Brandt1986}.
Moreover, \cite{Leipus2003} shows that
\[
\mathrm{Cov}(X_t,X_{t+h})
   \sim Ch^{\,3 - \alpha}, 
\qquad h, \to \infty\qquad\textrm{for some positive constant }C.
\]
Table \ref{tab:three_tests_rejection_rates_renewal_rcar} compares the power of the three tests across different parameter configurations $(c,p,\alpha)$. Note that  as $p$ increases, unit-root regimes become more frequent, while values of $c$ close to one weaken the contraction effect. The results clearly show that $Q_n(s)$ substantially outperforms both $\tilde r$ and its prewhitened version.
The test $\tilde r$ exhibits virtually no power across all configurations, indicating a failure to detect stationarity in this setting. Its prewhitened version performs well only in the very favorable scenario where unit-root regimes are rare (small $p$), contraction is strong ($c$ well below one), and the sample size is large.
The performance of $Q_n(s)$ is primarily driven by the parameters $p$ and $c$. 
When contracting regimes are sufficiently frequent ( $p\le0.5$) and  $c\le0.5$, the power of $Q_n(s)$ increases rapidly with the sample size, reaching values close to one. When $p$ is small (e.g., $p=0.2$), the power remains high even for relatively large values of $c$, such as $c=0.8$. In contrast, the parameter $\alpha$, which governs the tail behavior of the renewal durations, has only a limited impact on power.

\subsubsection{Processes with structural breaks}
Finally, we consider processes with structural breaks in the mean under the null hypothesis. In this setting, nonstationarity arises from deterministic level shifts rather than from stochastic integration. Table~\ref{tab:three_tests_break} reports the rejection rates when sampling from the structural break model
\[
X_t=
\begin{cases}
Y_t, & t=1,\ldots,n/2,\\[6pt]
Y_t+\delta, & t=n/2+1,\ldots,n,
\end{cases}
\]
where $Y_t$ is a stationary FARIMA(0,$d$,0).\\
When $\delta=1$, the mean shift is smaller than the variance of $Y_t$ given by
\[
\textrm{Var}(Y_t)=\frac{\Gamma(1-2d)}{\Gamma(1-d)^2} > 1(=\delta) \qquad \text{for } d>0.
\]
As a result, the change is not systematically detected by either test. By contrast, when $\delta=1.5$ or $2$, the statistic $Q_n(s)$ detects the change, and the rejection probability of nonstationarity falls below the nominal level, as expected. The test based on $\tilde r$ still fails to do so in several cases. The prewhitened version improves the performance of $\tilde r$, particularly for moderate sample size ($n=500$), but its performance deteriorates sharply for the larger sample ($n=2000$). Overall, these results indicate that $Q_n(s)$ is able to detect this type of nonstationarity generated by structural changes in the mean.
\begin{table}[H]
\centering
\caption{Nonstationarity empirical rejection rates for the three tests $Q_n(s)$, $\tilde r$, and prewhitened $\tilde r$ under FARIMA($p$,d,0) design, $\phi$ being the AR coefficient.  We take $s=10$ for $n=500$,   and $s=25$ for $n=2000$. We used 2000 Monte-Carlo replications. Bold numbers indicate the best performance: closest to $0.05$ when $d=0.5$, and closest to $1$ when $d<0.5$.}
\label{tab:three_tests_rejection_rates_AR3}
\begin{tabular}{cccccc}
\hline
$\phi$ & $n$ & $d$ & $Q_n(s)$ & $\tilde r$ & PW $\tilde r$\\
\hline
\multirow{12}{*}{0.0} & \multirow{6}{*}{500} & 0.5 & 0.092 & \textbf{0.072} & 0.021 \\
\cdashline{3-6}
 &  & 0.4 & 0.353 & \textbf{0.874} & 0.075 \\
 &  & 0.3 & 0.733 & \textbf{1.000} & 0.078 \\
 &  & 0.2 & 0.913 & \textbf{1.000} & 0.296 \\
 &  & 0.1 & 0.989 & \textbf{1.000} & 0.717 \\
 &  & 0 & \textbf{1.000} & \textbf{1.000} & 0.956 \\
\cline{2-6}
 & \multirow{6}{*}{2000} & 0.5 & 0.069 & \textbf{0.054} & 0.010 \\
\cdashline{3-6}
 &  & 0.4 & 0.542 & \textbf{1.000} & 0.106 \\
 &  & 0.3 & 0.872 & \textbf{1.000} & 0.808 \\
 &  & 0.2 & 0.993 & \textbf{1.000} & 1.000 \\
 &  & 0.1 & \textbf{1.000} & \textbf{1.000} & \textbf{1.000} \\
 &  & 0 & \textbf{1.000} & \textbf{1.000} & \textbf{1.000} \\
\hline
\multirow{12}{*}{0.2} & \multirow{6}{*}{500} & 0.5 & \textbf{0.033} & 0.000 & 0.000 \\
\cdashline{3-6}
 &  & 0.4 & \textbf{0.208} & 0.002 & 0.002 \\
 &  & 0.3 & \textbf{0.574} & 0.352 & 0.112 \\
 &  & 0.2 & 0.821 & \textbf{0.984} & 0.190 \\
 &  & 0.1 & 0.956 & \textbf{1.000} & 0.472 \\
 &  & 0 & 0.997 & \textbf{1.000} & 0.846 \\
\cline{2-6}
 & \multirow{6}{*}{2000} & 0.5 & \textbf{0.029} & 0.000 & 0.000 \\
\cdashline{3-6}
 &  & 0.4 & \textbf{0.369} & 0.000 & 0.038 \\
 &  & 0.3 & 0.802 & \textbf{0.876} & 0.580 \\
 &  & 0.2 & 0.974 & \textbf{1.000} & 0.997 \\
 &  & 0.1 & \textbf{1.000} & \textbf{1.000} & \textbf{1.000} \\
 &  & 0 & \textbf{1.000} & \textbf{1.000} & \textbf{1.000} \\
\hline
\multirow{12}{*}{0.5} & \multirow{6}{*}{500} & 0.5 & \textbf{0.005} & 0.000 & 0.000 \\
\cdashline{3-6}
 &  & 0.4 & \textbf{0.040} & 0.000 & 0.000 \\
 &  & 0.3 & \textbf{0.226} & 0.000 & 0.000 \\
 &  & 0.2 & \textbf{0.549} & 0.000 & 0.004 \\
 &  & 0.1 & \textbf{0.808} & 0.005 & 0.062 \\
 &  & 0 & \textbf{0.919} & 0.388 & 0.295 \\
\cline{2-6}
 & \multirow{6}{*}{2000} & 0.5 & \textbf{0.002} & 0.000 & 0.000 \\
\cdashline{3-6}
 &  & 0.4 & \textbf{0.092} & 0.000 & 0.000 \\
 &  & 0.3 & \textbf{0.521} & 0.000 & 0.067 \\
 &  & 0.2 & \textbf{0.855} & 0.000 & 0.720 \\
 &  & 0.1 & 0.978 & 0.002 & \textbf{0.997} \\
 &  & 0 & \textbf{1.000} & 0.944 & \textbf{1.000} \\
\hline
\end{tabular}
\end{table}
\begin{table}[H]
\centering
\caption{Nonstationarity empirical rejection rates for the three tests $Q_n(s)$, $\tilde r$, and prewhitened $\tilde r$ under the aggregated RCAR(1) design.  We take $s=10$ for $n=500$,   and $s=25$ for $n=2000$. We used 2000 Monte-Carlo replications. Bold numbers indicate the best performance: closest to $0.05$ when $d=0.5$, and closest to $1$ when $d<0.5$.}
\label{tab:three_tests_rejection_rates_arcar}
\begin{tabular}{cccccc}
\hline
$a$ & $n$ & $d$ & $Q_n(s)$ & $\tilde r$ & PW $\tilde r$\\
\hline
\multirow{12}{*}{1} & \multirow{6}{*}{500} & 0.5 & \textbf{0.000} & \textbf{0.000} & \textbf{0.000} \\
\cdashline{3-6}
 &  & 0.4 & 0.562 & \textbf{0.958} & 0.076 \\
 &  & 0.3 & 0.778 & \textbf{0.997} & 0.122 \\
 &  & 0.2 & 0.893 & \textbf{1.000} & 0.202 \\
 &  & 0.1 & 0.947 & \textbf{1.000} & 0.370 \\
 &  & 0 & 0.969 & \textbf{1.000} & 0.502 \\
\cline{2-6}
 & \multirow{6}{*}{2000} & 0.5 & 0.000 & 0.000 & \textbf{0.006} \\
\cdashline{3-6}
 &  & 0.4 & 0.754 & \textbf{1.000} & 0.378 \\
 &  & 0.3 & 0.911 & \textbf{1.000} & 0.784 \\
 &  & 0.2 & 0.969 & \textbf{1.000} & 0.958 \\
 &  & 0.1 & 0.988 & \textbf{1.000} & 0.990 \\
 &  & 0 & 0.998 & \textbf{1.000} & 0.999 \\
\hline
\multirow{12}{*}{2} & \multirow{6}{*}{500} & 0.5 & \textbf{0.000} & \textbf{0.000} & \textbf{0.000} \\
\cdashline{3-6}
 &  & 0.4 & \textbf{0.221} & 0.071 & 0.033 \\
 &  & 0.3 & \textbf{0.482} & 0.393 & 0.141 \\
 &  & 0.2 & 0.691 & \textbf{0.803} & 0.202 \\
 &  & 0.1 & 0.816 & \textbf{0.966} & 0.182 \\
 &  & 0 & 0.881 & \textbf{0.998} & 0.203 \\
\cline{2-6}
 & \multirow{6}{*}{2000} & 0.5 & 0.000 & 0.000 & \textbf{0.001} \\
\cdashline{3-6}
 &  & 0.4 & \textbf{0.465} & 0.117 & 0.084 \\
 &  & 0.3 & 0.743 & \textbf{0.782} & 0.387 \\
 &  & 0.2 & 0.908 & \textbf{0.995} & 0.653 \\
 &  & 0.1 & 0.963 & \textbf{1.000} & 0.881 \\
 &  & 0 & 0.981 & \textbf{1.000} & 0.970 \\
\hline
\multirow{12}{*}{5} & \multirow{6}{*}{500} & 0.5 & 0.000 & 0.000 & \textbf{0.000} \\
\cdashline{3-6}
 &  & 0.4 & \textbf{0.023} & 0.000 & 0.000 \\
 &  & 0.3 & \textbf{0.074} & 0.000 & 0.000 \\
 &  & 0.2 & \textbf{0.188} & 0.000 & 0.000 \\
 &  & 0.1 & \textbf{0.350} & 0.000 & 0.001 \\
 &  & 0 & \textbf{0.521} & 0.005 & 0.008 \\
\cline{2-6}
 & \multirow{6}{*}{2000} & 0.5 & \textbf{0.000} & \textbf{0.000} & \textbf{0.000} \\
\cdashline{3-6}
 &  & 0.4 & \textbf{0.044} & 0.000 & 0.000 \\
 &  & 0.3 & \textbf{0.268} & 0.000 & 0.002 \\
 &  & 0.2 & \textbf{0.546} & 0.000 & 0.017 \\
 &  & 0.1 & \textbf{0.722} & 0.000 & 0.160 \\
 &  & 0 & \textbf{0.858} & 0.006 & 0.419 \\
\hline
\end{tabular}
\end{table}
\begin{table}[H]
\centering
\caption{Empirical rejection rates of the three tests $Q_n(s)$, $\tilde r$, and prewhitened $\tilde r$ under the renewal random-coefficient autoregressive model. The random coefficient takes values $A_j\in\{1,c\}$ with $P(A_j=1)=p$. We take $s=10$ for $n=500$ and $s=25$ for $n=2000$. We used 2000 Monte-Carlo replications. Bold numbers indicate the largest rejection rate among the three tests.}
\label{tab:three_tests_rejection_rates_renewal_rcar}
\begin{tabular}{ccccccc}
\hline
$c$ & $\alpha$ & $p$ & $n$ & $Q_n(s)$ & $\tilde r$ & PW $\tilde r$\\
\hline

\multirow{12}{*}{0.5}
& \multirow{6}{*}{3.1}
& \multirow{2}{*}{0.2}
& 500  & \textbf{0.758} & 0.002 & 0.027 \\
&  &  & 2000 & \textbf{0.946} & 0.000 & 0.834 \\
\cline{3-7}
&  & \multirow{2}{*}{0.5}
& 500  & \textbf{0.288} & 0.000 & 0.000 \\
&  &  & 2000 & \textbf{0.748} & 0.000 & 0.009 \\
\cline{3-7}
&  & \multirow{2}{*}{0.8}
& 500  & \textbf{0.004} & 0.000 & 0.000 \\
&  &  & 2000 & \textbf{0.011} & 0.000 & 0.000 \\
\cline{2-7}

& \multirow{6}{*}{3.5}
& \multirow{2}{*}{0.2}
& 500  & \textbf{0.774} & 0.002 & 0.028 \\
&  &  & 2000 & \textbf{0.983} & 0.000 & 0.939 \\
\cline{3-7}
&  & \multirow{2}{*}{0.5}
& 500  & \textbf{0.364} & 0.000 & 0.000 \\
&  &  & 2000 & \textbf{0.835} & 0.000 & 0.031 \\
\cline{3-7}
&  & \multirow{2}{*}{0.8}
& 500  & \textbf{0.004} & 0.000 & 0.000 \\
&  &  & 2000 & \textbf{0.032} & 0.000 & 0.000 \\
\hline

\multirow{12}{*}{0.8}
& \multirow{6}{*}{3.1}
& \multirow{2}{*}{0.2}
& 500  & \textbf{0.144} & 0.000 & 0.000 \\
&  &  & 2000 & \textbf{0.692} & 0.000 & 0.000 \\
\cline{3-7}
&  & \multirow{2}{*}{0.5}
& 500  & \textbf{0.003} & 0.000 & 0.000 \\
&  &  & 2000 & \textbf{0.136} & 0.000 & 0.000 \\
\cline{3-7}
&  & \multirow{2}{*}{0.8}
& 500  & \textbf{0.000} & \textbf{0.000} & \textbf{0.000} \\
&  &  & 2000 & \textbf{0.000} & \textbf{0.000} & \textbf{0.000} \\
\cline{2-7}

& \multirow{6}{*}{3.5}
& \multirow{2}{*}{0.2}
& 500  & \textbf{0.142} & 0.000 & 0.000 \\
&  &  & 2000 & \textbf{0.725} & 0.000 & 0.000 \\
\cline{3-7}
&  & \multirow{2}{*}{0.5}
& 500  & \textbf{0.011} & 0.000 & 0.000 \\
&  &  & 2000 & \textbf{0.163} & 0.000 & 0.000 \\
\cline{3-7}
&  & \multirow{2}{*}{0.8}
& 500  & \textbf{0.000} & \textbf{0.000} & \textbf{0.000} \\
&  &  & 2000 & \textbf{0.000} & \textbf{0.000} & \textbf{0.000} \\
\hline

\end{tabular}
\end{table}
\begin{table}[H]
\centering
\caption{Empirical rejection rates under a structural-break design with break size $\delta \in \{1,1.5,2\}$. The three tests are $Q_n(s)$, $\tilde r$, and prewhitened $\tilde r$.  We take $s=10$ for $n=500$,   and $s=25$ for $n=2000$. We used 2000 Monte-Carlo replications. Bold numbers indicate rejection rates closest to $0$.}
\label{tab:three_tests_break}
\begin{tabular}{cccccc}
\hline
$\delta$ & $n$ & $d$ & $Q_n(s)$ & $\tilde r$ & PW $\tilde r$\\
\hline
\multirow{8}{*}{1} & \multirow{4}{*}{500} & 0.4 & 0.159 & 0.799 & \textbf{0.076} \\
 &  & 0.3 & 0.247 & 1.000 & \textbf{0.036} \\
 &  & 0.2 & 0.288 & 1.000 & \textbf{0.090} \\
 &  & 0.1 & \textbf{0.256} & 1.000 & \textbf{0.256} \\
\cline{2-6}
 & \multirow{4}{*}{2000} & 0.4 & 0.268 & 1.000 & \textbf{0.062} \\
 &  & 0.3 & \textbf{0.404} & 1.000 & 0.553 \\
 &  & 0.2 & \textbf{0.448} & 1.000 & 0.993 \\
 &  & 0.1 & \textbf{0.495} & 1.000 & 1.000 \\
\hline
\multirow{8}{*}{1.5} & \multirow{4}{*}{500} & 0.4 & 0.061 & 0.679 & \textbf{0.045} \\
 &  & 0.3 & 0.044 & 0.997 & \textbf{0.016} \\
 &  & 0.2 & 0.022 & 1.000 & \textbf{0.019} \\
 &  & 0.1 & \textbf{0.000} & 1.000 & 0.027 \\
\cline{2-6}
 & \multirow{4}{*}{2000} & 0.4 & 0.088 & 1.000 & \textbf{0.030} \\
 &  & 0.3 & \textbf{0.078} & 1.000 & 0.300 \\
 &  & 0.2 & \textbf{0.021} & 1.000 & 0.915 \\
 &  & 0.1 & \textbf{0.004} & 1.000 & 1.000 \\
\hline
\multirow{8}{*}{2} & \multirow{4}{*}{500} & 0.4 & \textbf{0.011} & 0.459 & 0.026 \\
 &  & 0.3 & \textbf{0.003} & 0.985 & 0.006 \\
 &  & 0.2 & \textbf{0.000} & 1.000 & 0.005 \\
 &  & 0.1 & \textbf{0.000} & 1.000 & 0.003 \\
\cline{2-6}
 & \multirow{4}{*}{2000} & 0.4 & 0.016 & 0.999 & \textbf{0.014} \\
 &  & 0.3 & \textbf{0.002} & 1.000 & 0.081 \\
 &  & 0.2 & \textbf{0.000} & 1.000 & 0.485 \\
 &  & 0.1 & \textbf{0.000} & 1.000 & 0.927 \\
\hline
\end{tabular}
\end{table}
\clearpage
\section{Proofs}\label{proofs}
\subsection{ Proof of Lemma \ref{invertible}}
\begin{proof}
\medskip
\noindent
Let $a=(a_1,\ldots,a_s)\in\mathbb{R}^s\setminus\{0\}$ and define
\[
\Phi(x) = \sum_{j=1}^{s} a_j \cos(2\pi jx), \qquad x\in(0,1),
\]
extended by $0$ outside $(0,1)$. Clearly, $\displaystyle \int_{\mathbb{R}}\Phi(x)\,dx=0$. \\We consider three cases according to the value of $d$.
\\  
\textbf{Case 1:} $\mathbf{d=\tfrac{1}{2}}$.
Using the integral representation
\[
-\log r = \int_0^\infty \frac{e^{-rs}-e^{-s}}{s}\,ds,
\]
we have
\[
a'\Sigma^{(c)}a = \iint_{\mathbb{R}^2} \Phi(x)\Phi(y)
\!\int_0^\infty \frac{e^{-s|x-y|}-e^{-s}}{s}\,ds\,dx\,dy.
\]
Denote $\hat f$  the Fourier transform of $f$ and $f*g$ the convolution of $f$ and $g$.
Interchanging the order of integration gives
\[
a'\Sigma^{(c)}a = \int_0^\infty \frac{1}{s}
\Bigl\langle \Phi,\, \Phi*K_s \Bigr\rangle\,ds,
\]
where $K_s(t)=e^{-s|t|}$ and $\widehat{K_s}(\xi)=\dfrac{2s}{s^2+\xi^2}$,  
By the convolution theorem and the Plancherel identity,
\[
\langle \Phi,\,\Phi*K_s\rangle
= \frac{1}{2\pi}<\widehat\Phi,\widehat\Phi\widehat K_s>=\frac{1}{2\pi}\int_{\mathbb{R}} |\widehat{\Phi}(\xi)|^2 \,\widehat{K_s}(\xi)\,d\xi.
\]
Hence,
\[
a'\Sigma^{(c)}a
= \frac{1}{\pi}\int_{\mathbb{R}} |\widehat{\Phi}(\xi)|^2
\Bigl(\int_0^\infty \frac{ds}{s^2+\xi^2}\Bigr)\,d\xi
= \frac{1}{2}\int_{\mathbb{R}}\frac{|\widehat{\Phi}(\xi)|^2}{|\xi|}\,d\xi.
\]
The integral is finite because $\Phi$ has compact support and
\[
\widehat{\Phi}(0)=\int_0^1\Phi(x)\,dx=0,
\]
so that $\widehat{\Phi}$ is continuously differentiable with
$|\widehat{\Phi}(\xi)|\le C|\xi|$ near $0$ and $|\widehat{\Phi}(\xi)|=O(|\xi|^{-1})$ as $|\xi|\to\infty$.  
Since $\Phi\not\equiv0$, the integral is strictly positive, proving positive definiteness.  
\medskip
\noindent
\textbf{Case 2:} $\mathbf{\tfrac{1}{2}< d<\tfrac{3}{2}}$.

Let $\beta=2d-1\in(0,2)$.  Using (3.823) of  \cite{gradshteyn2007tables}, and the fact that for any non integer value $z$,
\[
\Gamma(-z) = -\,\frac{\pi}{z\,\Gamma(z)\,\sin(\pi z)}.
\]
(the case $\beta=1$ is obtained directly as $\Gamma$  function is not defined for nonpositive integers), we obtain
\[
|u|^{\beta}
= c_\beta \int_0^\infty (1-\cos(2\pi us))\,s^{-1-\beta}\,ds,
\quad
c_\beta=\frac{2^{1-\beta}\Gamma(1+\beta)\sin(\pi\beta/2)}{\pi^{1+\beta}}>0,
\]
and the fact that $\int_0^1\Phi(x)\,dx=0$, by Fubini Theorem, we get
\[
a'\Sigma^{(c)}a
= c_\beta \int_0^\infty s^{-1-\beta}\,|\widehat{\Phi}(2\pi s)|^2\,ds > 0.
\]

\medskip
\noindent
\textbf{Case 3:} $\mathbf{-\tfrac{1}{2}<d<\tfrac{1}{2}}$.
Here $\beta=2d+1\in(0,2)$, and we can go back to case 2.

For all three cases above, the same reasoning applies to $\Sigma^{(s)}$.
Thus, in all cases, $\Sigma^{(c)}$ and $\Sigma^{(s)}$ are positive definite, and therefore $\Sigma(d)$ is positive definite.
\end{proof}
\subsection{Proof of Proposition \ref{unit root1}}
\begin{proof}
(i)  It is an extension of Proposition 3.3.1 (ii) of \cite{MR2977317} to the case $d=-1/2.$ 
Using their notations and techniques, with $(\log(n))^{-1}$ instead of $n^{-1-2d}$, we can write for small $\epsilon>0$, with $S_n=Y_1+\cdots+Y_n$,
\begin{eqnarray*}\lefteqn{
(\log(n))^{-1}\textrm{Var}\left(S_n\right)}\\
&&=(\log(n))^{-1}\int_{-\pi}^\pi\left(\frac{\sin(n\lambda/2)}{\sin(\lambda/2)}\right)^2f(\lambda)d\lambda \\
&&=2(\log(n))^{-1}\int_0^\epsilon\left(\frac{\sin(n\lambda/2)}{\sin(\lambda/2)}\right)^2f(\lambda)d\lambda +2(\log(n))^{-1}\int_\epsilon^\pi\left(\frac{\sin(n\lambda/2)}{\sin(\lambda/2)}\right)^2f(\lambda)d\lambda 
\end{eqnarray*}
The second term clearly goes to zero as $n\to\infty.$ Writing 
$$f(\lambda)=\left(\frac{\sin(\lambda/2)}{(\lambda/2)}\right)^2(\lambda/2) g(\lambda),$$ with $g(\lambda)$ converging to $2c_f$ as $\lambda\to0$, we get
\begin{eqnarray*}\lefteqn{
2(\log(n))^{-1}\int_0^\epsilon\left(\frac{\sin(n\lambda/2)}{\sin(\lambda/2)}\right)^2f(\lambda)d\lambda
=4(\log(n))^{-1}\int_0^{\epsilon n}\left(\frac{\sin^2(u/2)}{u}\right)g(\frac{u}{n})du}\\
    &&=2(\log(n))^{-1}\int_0^{\epsilon n}\left(\frac{1-\cos u}{u}\right)g(\frac{u}{n})du\\
    &&=
    2(\log(n))^{-1}\int_0^1\left(\frac{1-\cos u}{u}\right)g(\frac{u}{n})du+
    2(\log(n))^{-1}\int_1^{\epsilon n}\left(\frac{1-\cos u}{u}\right)g(\frac{u}{n})du.
\end{eqnarray*}
The first term goes to zero as $n\to\infty$ since the integrand is uniformly bounded in $u$  and $n$. Writing $g(u)=2c_f+h(u)$ where $h(u)\to0$ as $u\to0$, we get
\begin{eqnarray*}\lefteqn{
(\log(n))^{-1}\int_1^{\epsilon n}\left(\frac{1-\cos(u)}{u}\right)g(\frac{u}{n})du=(\log(n))^{-1}\int_{1/n}^{\epsilon}\left(\frac{1-\cos(nu)}{u}\right)(2c_f+h(u))du}\\
&&=2c_f-2c_f(\log(n))^{-1}\int_{1/n}^{\epsilon}\frac{\cos(nu)}{u}du+(\log(n))^{-1}\int_{1/n}^{\epsilon}\left(\frac{1-\cos(nu)}{u}\right)h(u)du.
\end{eqnarray*}
One integration by parts shows that, as $n\to\infty$,
$$
(\log(n))^{-1}\int_{1/n}^{\epsilon}\frac{\cos(nu)}{u}du\to0.
$$
Also, choosing $\epsilon$ small enough so that $h(u)$ is bounded (by $M$ say) for $u\le\epsilon$, we get
\begin{eqnarray*}\lefteqn{
\left\vert(\log(n))^{-1}\int_{1/n}^{\epsilon}\left(\frac{1-\cos(nu)}{u}\right)h(u)du\right\vert}\\
&&\le2\left(\underset{1/n\le u\le 1/\log(n)}{\max}\vert h(u)\vert\right)(\log(n))^{-1}\int_{1/n}^{1/\log(n)}\frac{du}{u}+2M(\log(n))^{-1}\int_{1/\log(n)}^\epsilon\frac{du}{u}\to0,
\end{eqnarray*}
as $n\to\infty.$\\
(ii) The proof is essentially identical to the proof of Proposition~3.2.2 of \cite{MR2977317}, which establishes the behaviour of the spectral density at the origin for $d \in (-1/2,1/2)$.  
The only modification is the following (we use their notation).

Write the transfer function as
\begin{eqnarray*}\lefteqn{
A(\lambda)=\sum_{j=0}^\infty a_je^{-\i j\lambda}}\\
&&= c_a \sum_{j=1}^\infty e^{-\mathrm{i} j \lambda} j^{-3/2}
+ \left( a_0 + \sum_{j=1}^\infty e^{-\mathrm{i} j \lambda} (a_j - c_a j^{-3/2}) \right)
=: c_a A_1(\lambda) + A_2(\lambda).
\end{eqnarray*}
Since $A(0)=0$, we write
\[
A(\lambda)
= c_a \bigl[ A_1(\lambda) - A_1(0) \bigr]
+ \bigl[ A_2(\lambda) - A_2(0) \bigr].
\]

Using the fact that for $0<\alpha<1$ 
\[
\bigl| e^{\mathrm{i} j \lambda} - 1 \bigr|
\le \min\bigl( |j\lambda|^2,\; |j\lambda| \bigr)\le \vert j\lambda\vert^{\alpha+1}
\]
we obtain (with $0<\alpha<1/2$)
\[
\bigl| A_2(\lambda) - A_2(0) \bigr|
\le |\lambda|^{1+\alpha} \sum_{j=1}^\infty j^{-3/2+\alpha}
= o(|\lambda|),\qquad\textrm{as }\lambda\to0.
\]
The remainder of the proof proceeds exactly as in the proof of Proposition~3.2.2 in Giraitis et al.\end{proof}
\subsection{Proof of Theorem \ref{finite dim}}
\begin{proof} 
It is based on Theorem 4.3.2 of \cite{MR2977317} and our Lemma \ref{invertible}.\\
{\bf Case of $\mathbf{1/2\le d<3/2.}$} Let $\lambda_1,\ldots,\lambda_s$ be the $s$ first  Fourier frequencies. 
Let $Y_k=X_k-X_{k-1}$ be the stationary differenced process, i.e. $Y_k$ is a zero-mean stationary process with memory parameter $d-1\in[-1/2,1/2)$. We have
$$
X_k=\sum_{t=1}^kY_t+X_0:=S_k(Y)+X_0.
$$
Since, for any positive integer $u,s$ and fixed $j=1,\ldots,s$
\begin{equation}\label{sumzero2}
\sum_{t=1}^ue^{2\pi \i tj/u}=0,
\end{equation}
the real and imaginary parts of the DFT built from $(X_k)_{1\le k\le n}$ and $(S_k(Y))_{1\le k\le n}$ are the same. \\
{\bf Sub-case $\mathbf{d=1/2}$:} 
Using Proposition \ref{unit root1}, we obtain that 
$$
\textrm{Var}(S_n(Y))\sim \delta(-1/2)\log(n)
$$
where $\delta(-1/2)$ is defined in (\ref{const delta}).
Also, using the fact that $ab=1/2\left(a^2+b^2-(a-b)^2\right)$, the stationarity of the increments of $S_k(Y)$ and (\ref{sumzero2}), we obtain that for every fixed $i,j=1,\ldots,s$ and as $n\to\infty$, 
\begin{eqnarray}\label{of course}
\lefteqn{
    \textrm{Cov}\left[\left(\frac{1}{n}\sum_{k=1}^n\cos(k\lambda_{i})S_k(Y)\right),\left(\frac{1}{n}\sum_{k=1}^n\cos(k\lambda_{j})S_k(Y)\right)\right]}\nonumber\\
    &&=\frac{1}{n^2}\sum_{k=1}^n\sum_{k'=1}^n\cos\left(\frac{2\pi ik}{n}\right)\cos\left(\frac{2\pi jk'}{n}\right)\mathbb{E}\left(S_k(Y)S_{k'}(Y)\right)\nonumber\\ 
    &&=\frac{1}{2}\frac{1}{n^2}\sum_{k=1}^n\sum_{k'=1}^n\cos\left(\frac{2\pi ik}{n}\right)\cos\left(\frac{2\pi jk'}{n}\right)(-\mathbb{E}(S^2_{\vert k-k'\vert}(Y)))\nonumber\\
    &&\sim\frac{\delta(-1/2)}{2}\frac{1}{n^2}\sum_{k=1}^n\sum_{k'=1}^n\cos\left(\frac{2\pi ik}{n}\right)\cos\left(\frac{2\pi jk'}{n}\right)\left(-\log\left(\frac{\vert k-k'\vert}{n}\right)\right)\nonumber\\
&&\to\frac{\delta(-1/2)}{2}\int_{[0,1]^2}\cos(2\pi ix)\cos(2\pi jy)(-\log(\vert x- y\vert)dxdy=\Sigma_{ij}^{(c)}(1/2)
\end{eqnarray}
and similarly
\begin{eqnarray*}\lefteqn{
    \textrm{Cov}\left[\left(\frac{1}{n}\sum_{k=1}^n\sin(k\lambda_{i})S_k(Y)\right),\left(\frac{1}{n}\sum_{k=1}^n\sin(k\lambda_{j})S_k(Y)\right)\right]}\\
&&\to\frac{\delta(-1/2)}{2}\int_{[0,1]^2}\sin(2\pi ix)\sin(2\pi jy)(-\log(\vert x- y\vert)dxdy,
\end{eqnarray*}
\begin{eqnarray*}\lefteqn{
    \textrm{Cov}\left[\left(\frac{1}{n}\sum_{k=1}^n\cos(k\lambda_{i})S_k(Y)\right),\left(\frac{1}{n}\sum_{k=1}^n\sin(k\lambda_{j})S_k(Y)\right)\right]}\\
   && \to\frac{\delta(-1/2)}{2}\int_{[0,1]^2}\cos(2\pi ix)\sin(2\pi jy)(-\log(\vert x- y\vert)dxdy=0.
   \end{eqnarray*}
   We can write,   with the convention that $a_h=0$ for $h<0$
\begin{eqnarray}\lefteqn{
\frac{1}{n}\sum_{k=1}^n\cos(k\lambda_{i})S_k(Y)
=\sum_{i=1}^n\left(\frac{1}{n}\sum_{k=1}^{i-1}\cos(k\lambda_j)\right)Y_i}\label{develop}\\
&&=
\sum_{i=1}^n\left(\frac{1}{n}\sum_{k=i}^n\cos(k\lambda_j)\right)
\sum_{u=-\infty}^ia_{i-u}\epsilon_u\nonumber\\
&&=\sum_{u=-\infty}^n\left(\
\frac{1}{n}\sum_{k=1}^n\left(\sum_{i=1}^ka_{i-u}\right)\cos(k\lambda_j)\right)\epsilon_u
:=\sum_{u=-\infty}^nd^{(c)}_{n,u}\epsilon_u.\label{combine}
\end{eqnarray}
For $u\le0$,  since
$$
\sum_{i=0}^\infty a_i=0,
$$
$$
d^{(c)}_{n,u}=\left(\
\frac{1}{n}\sum_{k=1}^n\left(\sum_{i=0}^{k-u}a_i\right)\cos(k\lambda_j)\right)=
-\frac{1}{n}\sum_{k=1}^n\left(\sum_{i=k-u+1}^\infty a_i\right)\cos(k\lambda_j)
$$
and hence, since $d-1=-1/2$,  $\vert a_k\vert\sim c k^{-3/2}$, as $k\to\infty$, for some positive constant $c$,
so that
$$
\sum_{i=\ell}^\infty \vert a_i\vert \sim2c\ell^{-1/2},\qquad\textrm{as }\ell\to\infty,
$$
and hence,
$$
\vert d^{(c)}_{n,u}\vert\le 2\vert c\vert n^{-1/2}\to0,\qquad n\to\infty.
$$
Now consider the case $0<u\le n$.
$$
d^{(c)}_{n,u}=\frac{1}{n}\sum_{k=u}^n\left(\sum_{i=0}^{k-u}a_i\right)\cos(k\lambda_j)
$$
Here again, we get
$$
\vert d^{(c)}_{n,u}\vert\le 2\vert c\vert n^{-1/2}\to0\qquad\textrm{as }n\to\infty,
$$
and hence we have shown that 
\begin{equation}\label{uniform u}
   \underset{u\le n}{\sup}\,\,d^{(c)}_{n,u}\to0 \qquad\textrm{as }n\to\infty,
\end{equation}
and using \eqref{of course}, as $n\to\infty$, we get
$$
\sum_{u=-\infty}^n\left(d^{(c)}_{n,u}\right)^2=\mathbb{E}\left[\left(\frac{1}{n}\sum_{k=1}^n\cos(k\lambda_{j})S_k(Y)\right)^2\right]\to\Sigma^{(c)}_{jj}(1/2)>0.
$$
and similarly we obtain 
$$
d^{(s)}_{n,u}:=\frac{1}{n}\sum_{k=1}^n\left(\sum_{i=1}^ka_{i-u}\right)\sin(k\lambda_j)\to0\quad\textrm{uniformly in }u\textrm{ as }n\to\infty.
$$
and that
$$
\sum_{u=-\infty}^n\left(d^{(s)}_{n,u}\right)^2=\mathbb{E}\left[\left(\frac{1}{n}\sum_{k=1}^n\sin(k\lambda_{j})S_k(Y)\right)^2\right]\to\Sigma^{(s)}_{jj}(1/2)>0.
$$
Since $\Sigma(1/2)$ is invertible by virtue of Lemma \ref{invertible}, and therefore by Theorem 4.3.2 of \cite{MR2977317},  we obtain, as $n\to\infty$,
\begin{eqnarray*}\lefteqn{
\Bigg[\left(\frac{1}{n}\sum_{k=1}^n\cos(k\lambda_{1})S_k(Y)\right),
\cdots,\left(\frac{1}{n}\sum_{k=1}^n\cos(k\lambda_{s})S_k(Y)\right),}\\
&& \left(\frac{1}{n}\sum_{k=1}^n\sin(k\lambda_{i})S_k(Y)\right),\cdots,\left(\frac{1}{n}\sum_{k=1}^n\sin(k\lambda_{s})S_k(Y)\right)\Bigg]\overset{d}{\rightarrow}\mathcal{N}\left(0,\Sigma(1/2)\right). 
\end{eqnarray*}

{\bf Sub-case $\mathbf{1/2<d<3/2}$:}  The differenced process $Y_k$ is a zero-mean linear process with memory parameter $d-1\in(-1/2,1/2)$, so that, by Propositions  3.2.1 and 3.3.1 of \cite{MR2977317},
\begin{equation}\label{giraitis stationary}
\textrm{Var}(S_n(Y))\sim\delta(d-1)n^{2(d-1)+1}=\delta(d-1) n^{2d-1}.
\end{equation}
Then we get
\begin{eqnarray}\label{cumsum}\lefteqn{
    \textrm{Cov}\left[\left(\frac{1} {n^{1/2+d}}\sum_{k=1}^n\cos(k\lambda_{i})S_k(Y)\right),\left(\frac{1}{n^{1/2+d}}\sum_{k=1}^n\cos(k\lambda_{j})S_k(Y)\right)\right]}\nonumber\\
    &&=\frac{1}{n^{1+2d}}\sum_{k=1}^n\sum_{k'=1}^n\cos\left(\frac{2\pi ik}{n}\right)\cos\left(\frac{2\pi jk'}{n}\right)\mathbb{E}\left(S_k(Y)S_{k'}(Y)\right)\nonumber\\ 
    &&=\frac{1}{2}\frac{1}{n^{1+2d}}\sum_{k=1}^n\sum_{k'=1}^n\cos\left(\frac{2\pi ik}{n}\right)\cos\left(\frac{2\pi jk'}{n}\right)(-\mathbb{E}(S^2_{\vert k-k'\vert}(Y))\nonumber\\
    &&\sim\frac{\delta(d-1)}{2}\frac{1}{n^{1+2d}}\sum_{k=1}^n\sum_{k'=1}^n\cos\left(\frac{2\pi ik}{n}\right)\cos\left(\frac{2\pi jk'}{n}\right)\left(-\left(\frac{1+\vert k-k'\vert}{n}\right)^{2d-1}\right)n^{2d-1}\nonumber\\
&&\to-\frac{\delta(d-1)}{2}\int_{[0,1]^2}\cos(2\pi ix)\cos(2\pi jy)(\vert x- y\vert^{2d-1})dxdy=\Sigma_{ij}^{(c)}(d).
\end{eqnarray}
Since 
$$
\frac{1} {n^{1/2+d}}\sum_{k=1}^n\cos(k\lambda_{j})S_k(Y)=\frac{1} {n^{1/2+d}}\sum_{k=1}^n\left(\sum_{u=1}^{k-1}\cos(u\lambda_{j})\right)Y_k
$$
and that, because $d>1/2$,
$$
\underset{k\le n}{\sup}\vert z_{n,k}\vert:=\underset{k\le n}{\sup}\frac{1} {n^{1/2+d}}\left\vert\sum_{u=1}^{k-1}\cos(u\lambda_{j})\right\vert\to0
$$
and
$$
\sum_{k=1}^nz_{n,k}^2\le\frac{1}{n^{2d-1}}\to0,
$$
then by Theorem 4.3.2 of \cite{MR2977317} and Lemma \ref{invertible},  we complete the proof of   Theorem  \ref{finite dim}, when  $1/2\le d<3/2$.\\\\
{\bf Proof when $\mathbf{-1/2<d<1/2}$:} 
Before proceeding with the proof of this case, we note that \cite{MR1243575} established a related result for the covariance matrix of the DFT with with a different integral representation than  $\Sigma(d)$ under the rather restrictive assumption that the spectral density $f$ is of the form
\[
f(\lambda)=|\lambda|^{-2d}f^*(\lambda), \qquad f^* \text{ continuous and positive}.
\]
Their proof relies heavily on this condition, which we do not impose here. Instead, we assume linearity in order to establish the joint asymptotic normality of the DFTs. In particular, when $0<d<1/2$, our framework (\ref{positived}) allows for spectral densities that may be unbounded at frequencies other than the origin, provided that the singularity at zero remains dominant.\\
Without loss of generality, we can assume that $\mathbb{E}(X_1)=0$. Writing $X_k=S_k-S_{k-1}$ where, being the sum, $S_k$ is an integrated  process with memory parameter $1/2<d+1<3/2$, and using summation by parts, we obtain that 
\begin{eqnarray}\label{expansion}\lefteqn{
\sum_{k=1}^n\cos(k\lambda_{j})X_k=(\cos(\lambda_j))(S_n-X_0)-\sum_{k=1}^n\left[\cos((k+1)\lambda_j)-\cos(k\lambda_j)\right]S_k}\nonumber\\
&&=(\cos(\lambda_j))S_n+
\sum_{k=1}^n\left[\cos(k\lambda j)(1-\cos(\lambda j))+\sin(k\lambda j)\sin(\lambda j)\right]S_k.
\end{eqnarray}
Hence,  applying (\ref{cumsum}) with $d+1$ instead of $d$, we obtain that, as $n\to\infty$, (since  $\cos(\lambda_j)\sim1$, $1-\cos(\lambda j)\sim(2\pi j/n)^2/2$ and $\sin(\lambda j)\sim2\pi j/n$),
\begin{eqnarray}\label{match2}\lefteqn{
    \textrm{Cov}\left[\left(\frac{1} {n^{1/2+d}}\sum_{k=1}^n\cos(k\lambda_{i})X_k\right),\left(\frac{1}{n^{1/2+d}}\sum_{k=1}^n\cos(k\lambda_{j})X_k\right)\right]}\\
&&
\to\delta(d)\Bigg[1+\frac{1}{2}\int_0^1\left(2\pi i\sin(2\pi ix)+2\pi j\sin(2\pi jx)\right)\left(x^{2d+1}-(1-x)^{2d+1}\right)dx\nonumber\\
&&-\frac{(2\pi i)(2\pi j)}{2}\int_{[0,1]^2}\sin(2\pi ix)\sin(2\pi jy)\vert x- y\vert^{2d+1}dxdy\Bigg]\nonumber\\
&&=\delta(d)\Bigg[-1+(2d+1)\int_0^1x^{2d}\left(\cos(2\pi ix)+\cos(2\pi jx)\right)dx\nonumber\\
&&-\frac{(2\pi i)(2\pi j)}{2}\int_{[0,1]^2}\sin(2\pi ix)\sin(2\pi jy)\vert x- y\vert^{2d+1}dxdy\Bigg]
\end{eqnarray}
and similarly we obtain that
\begin{eqnarray*}\lefteqn{
    \textrm{Cov}\left[\left(\frac{1} {n^{1/2+d}}\sum_{k=1}^n\sin(k\lambda_{i})X_k\right),\left(\frac{1}{n^{1/2+d}}\sum_{k=1}^n\sin(k\lambda_{j})X_k\right)\right]}\\
&&\to-\delta(d)(2\pi^2 ij)\int_{[0,1]^2}\cos(2\pi ix)\cos(2\pi jy)\vert x- y\vert^{2d+1}dxdy.
\end{eqnarray*}
The sequence of coefficients 
$$
z^{(c)}_{n,k}:=\frac{\cos(2\pi jk)}{n^{1/2+d}}\qquad\textrm{and }z^{(s)}_{n,k}:=\frac{\sin(2\pi jk)}{n^{1/2+d}},\qquad k=1,\ldots,n,\qquad j=1,\ldots,s,
$$
in the sequence of random vectors in Theorem \ref{finite dim} clearly 
satisfy conditions (i)  (respectively (ii)) of Proposition 4.3.1 of  \cite{MR2977317} when $0<d<1/2$ (respectively $-1/2<d\le0$) and hence by Theorem 4.3.2 of the same reference, we obtain the result  of the Theorem \ref{finite dim} in the case of linear stationary processes. 
\end{proof}
\subsection{Proof of Theorem \ref{1 over f1}}
\begin{proof}
\noindent\textbf{Proof of part A.}
Recall that
\[
\bar X_n=\frac1n\sum_{k=1}^n X_k
=X_0+\sum_{j=1}^n\Big(1-\frac{j}{n}\Big)Y_j,
\]
so that
\[
\frac{\bar X_n}{d_n}
=\frac{X_0}{d_n}+\sum_{j=1}^n \frac{1-j/n}{d_n}\,Y_j,
\]
where $(Y_j)$ is a zero-mean stationary process with memory parameter $d-1.$ Moreover, $X_0/d_n\to0$ a.s., and
\begin{equation}\label{lindeberg1}
\max_{1\le j\le n}\Big|\frac{1-j/n}{d_n}\Big|
\le \frac{1}{d_n}.
\end{equation}
We now verify the asymptotic variance.\\
{\bf Case $\mathbf{d=1/2}$}
Using Proposition \ref{unit root1},   $\Var(X_k)\sim \delta(-1/2)\log k$ as $k\to\infty$, we have
\[
\Var\!\Big(\frac{\bar X_n}{\sqrt{\log(n)}}\Big)
=\frac{1}{n^2\log(n)}\sum_{k=1}^n\sum_{k'=1}^n \mathbb E(X_kX_{k'}).
\]
Using the identity
\[
\mathbb E(X_kX_{k'})
=\frac12\bigl(\mathbb E(X_k^2)+\mathbb E(X_{k'}^2)
-\mathbb E(X_{|k-k'|}^2)\bigr),
\]
we obtain
\[
\Var\!\Big(\frac{\bar X_n}{\sqrt{\log(n)}}\Big)
=\frac12\left[
\frac{2}{n\log(n)}\sum_{k=1}^n \mathbb E(X_k^2)
-\frac{1}{n^2\log(n)}\sum_{k=1}^n\sum_{k'=1}^n \mathbb E(X_{|k-k'|}^2)
\right].
\]
and
\begin{equation}\label{var log n}
\frac{1}{n\log(n)}\sum_{k=1}^n \mathbb E(X_k^2)=\frac{1}{n\log(n)}\sum_{k=1}^n\left(\textrm{Var}(X_k)+(\mathbb{E}(X_0))^2\right)\to \delta(-1/2),
\end{equation}
since $\sum_{h=1}^n \log h\sim n\log(n)$.
Also, we have
\[
\frac{1}{n^2\log(n)}\sum_{k,k'=1}^n \mathbb E(X_{|k-k'|}^2)
\sim \frac{2\delta(-1/2)}{n^2\log(n)}\sum_{h=1}^{n}(n-h)\log h.
\]
Writing
\[
\sum_{h=1}^{n}(n-h)\log h
=
n\sum_{h=1}^n \log h-\sum_{h=1}^n h\log h,
\]
and using  
$\sum_{h=1}^n h\log h\sim \frac{n^2}{2}\log(n)$, we get
\[
\frac{1}{n^2\log(n)}\sum_{h=1}^{n}(n-h)\log h \to \frac12,
\]
and therefore we get
\begin{equation}\label{1/2 variance}
\Var\!\Big(\frac{\bar X_n}{\sqrt{\log(n)}}\Big)\to \tfrac12\delta(-1/2).
\end{equation}
{\bf Case $\mathbf{1/2<d<3/2}$.} A similar  computation as above, with (\ref{giraitis stationary}) shows that
$$
\Var\!\Big(\frac{\bar X_n}{d_n}\Big)\to \tfrac1{2d+1}\delta(d-1).
$$
Combining these two convergences  with  (\ref{lindeberg1}) and using 
  Proposition~4.3.1 (i) and (ii) of Giraitis et al. concludes the proof of part A.\\
  {\bf Proof of Part B} Recall that here we have $d=1/2.$
\begin{enumerate}
    \item
Writing
\[
\sum_{k=1}^n X_k^2
=\sum_{k=1}^n (X_k-\bar X_n)^2+n\bar X_n^2.
\] Taking expectations and dividing by $n\log(n)$ gives
\[
\frac{1}{n\log(n)}\sum_{k=1}^n \mathbb E(X_k^2)
=
\frac{1}{n\log(n)}\sum_{k=1}^n \mathbb E(X_k-\bar X_n)^2
+\frac{\mathbb E(\bar X_n^2)}{\log(n)}.
\]
which converges to $(1/2)\delta(-1/2)$ according to \eqref{var log n}.
Moreover, we have from \eqref{1/2 variance},
\[
\frac{\mathbb E(\bar X_n^2)}{\log(n)}\to \tfrac12\delta(-1/2).
\]
Therefore we have
\[
\frac{1}{n\log(n)}\sum_{k=1}^n \mathbb E(X_k-\bar X_n)^2\to \tfrac12\delta(-1/2).
\]
Thus, to complete the proof, it suffices to show that
\begin{equation}\label{lim var}
\Var\!\left(
\frac{1}{n\log(n)}\sum_{k=1}^n (X_k-\bar X_n)^2
\right)\to0.
\end{equation}

Recall that if
\[
U=\sum_{i=0}^\infty a_i\varepsilon_i,
\qquad
V=\sum_{i=0}^\infty b_i\varepsilon_i,
\]
where $(\varepsilon_i)$ are i.i.d., centered, with finite fourth moment, then  the moment-cumulant relationship and the multilinearity of cumulants imply that
\begin{eqnarray*}
\Cov(U^2,V^2)&=&2\,\Cov(U,V)^2+\textrm{Cum}(U,U,V,V)\\
&=&2\,\Cov(U,V)^2+\eta_4\sum_{i=0}^\infty a_i^2 b_i^2,
\end{eqnarray*}
where $\eta_4=\textrm{cum}_4(\varepsilon_1)$ is the fourth cumulant of $\varepsilon_i$.

Now set
\[
U^{(k)}:=X_k-\bar X_n
=\sum_{s=-\infty}^n a_{s,n}^{(k)}\,\varepsilon_s,
\qquad k=1,\dots,n,
\]
where
\[
a_{s,n}^{(k)}
=\sum_{j=1}^n\Big(\frac{j}{n}-\mathbf 1_{\{j>k\}}\Big)a_{j-s}.
\]
In what follows, we assume without loss of generality that $\mathbb{E}(\epsilon_1^2)=1.$
Clearly, $|a_{s,n}^{(k)}|$ is uniformly bounded in $s,n,k$ since $(a_j)$ is absolutely summable.

\medskip
\noindent\emph{(i) Sum over $s\le0$.}
Since $a_j\sim j^{-3/2}$,
\[
\sum_{s=-\infty}^{0}\bigl|a_{s,n}^{(k)}\bigr|^2
\le
C\sum_{s=0}^{\infty}\left(\sum_{j=1}^{n}(j+s)^{-3/2}\right)^2
\le
C\sum_{s=1}^{\infty}\bigl(s^{-1/2}-(n+s)^{-1/2}\bigr)^2.
\]
Using $(u-v)^2\le u^2-v^2$ for $0\le v\le u$ gives
\begin{eqnarray*}\lefteqn{
\sum_{s=-\infty}^{-1}\bigl|a_{s,n}^{(k)}\bigr|^2}\\
&&\le
C\sum_{s=1}^{\infty}\Bigl(s^{-1}-(n+s)^{-1}\Bigr)
=
C\left(\sum_{s=1}^{n}\Bigl(s^{-1}-(n+s)^{-1}\Bigr)+\sum_{s=n+1}^{\infty}\Bigl(s^{-1}-(n+s)^{-1}\Bigr)\right).
\end{eqnarray*}
The first sum is bounded by $C\sum_{s=1}^{n}s^{-1}=O(\log(n))$, and for the second sum,
\[
\sum_{s=n+1}^{\infty}\Bigl(s^{-1}-(n+s)^{-1}\Bigr)
=\sum_{s=n+1}^\infty \frac{n}{s(n+s)}
\le n\sum_{s=n+1}^\infty \frac{1}{s^2}=O(1).
     \]
Hence
\[
\sum_{s=-\infty}^{0}\bigl|a_{s,n}^{(k)}\bigr|^2=O(\log(n)),
\]
uniformly in $k$.

\medskip
\noindent\emph{(ii) Sum over $k+1\le s\le n$.}
For $s\in\{k+1,\dots,n\}$,
\[
a_{s,n}^{(k)}=\sum_{j=s}^n\Big(\frac{j}{n}-1\Big)a_{j-s}
=\sum_{i=0}^{n-s}\Big(\frac{i+s}{n}-1\Big)a_i,
\]
Using the fact that $\sum_{i=0}^\infty a_i=0$,
the partial sums satisfy $\sum_{i=0}^m a_i\sim C m^{-1/2}$. Consequently,
\[
\sum_{s=k+1}^n \bigl(a_{s,n}^{(k)}\bigr)^2
\le
C\sum_{s=1}^n \left[\left(1-\frac{s}{n}\right)^{1/2}n^{-1/2}+\left(1-\frac{s}{n}\right)\left(1-\frac{s}{n}\right)^{-1/2}n^{-1/2}\right]^2\to C.
\]
\medskip
\noindent\emph{(iii) Sum over $0\le s\le k$.}
\[
\sum_{s=0}^{k}\left(\sum_{j=s}^{k}\frac{j}{n}\,a_{j-s}\right)^2
=
\sum_{s=0}^{k}\left(\sum_{i=0}^{k-s}\frac{i+s}{n}\,a_i\right)^2,
\]
\[
\sum_{s=0}^{k}\left(\sum_{i=0}^{k-s}\frac{i+s}{n}\,a_i\right)^2
\le
C\sum_{s=0}^{k}\left(\frac{1}{n}\Big((k-s)^{1/2}+s\,(k-s)^{-1/2}\Big)\right)^2=O(\log(n))
\]
uniformly in $k\le n$.
Now let us investigate
$$
\left[\textrm{Cov}\left(U^{(k)},U^{(k')}\right)\right]^2.
$$
Using the fact that $(a+b+c)^2\le 3(a^2+b^2+c^2)$, we get
\begin{eqnarray*}\lefteqn{
\left[\textrm{Cov}\left(U^{(k)},U^{(k')}\right)\right]^2}\\
&&\le3\left[\sum_{s=n+1}^\infty\left(\left(\sum_{j=1}^n\vert a_{j+s}\vert\right)^2\right)\right]^2\\
&&+3\left[\sum_{s=0}^n\left(\left(\sum_{j=1}^k\frac{j}{n}\vert a_{j+s}\vert+\sum_{j=k+1}^n\left(1-\frac{j}{n}\right)\vert a_{j+s}\vert\right)\left(\sum_{j=1}^{k'}\frac{j}{n}\vert a_{j+s}\vert+\sum_{j=k'+1}^n\left(1-\frac{j}{n}\right)\vert a_{j+s}\vert\right)\right)\right]^2\\
&&+3\left[\sum_{s=1}^n\left(\sum_{j=s}^n\left(\frac{j}{n}-\mathbf 1_{\{j>k\}}\right)a_{j-s}\right)
\left(\sum_{j=s}^n\left(\frac{j}{n}-\mathbf 1_{\{j>k'\}}\right)a_{j-s}\right)\right]^2\\
&&:=3\left(I^2_1(n)+I^2_2(n,k,k')+I^2_3(n,k,k')\right)
\end{eqnarray*}
We have
$$
I_1(n)\le 4\left(\sum_{s=n+1}^\infty\frac{n}{s(n+s)}\right)^2=O(1).
$$
Also, for all $k,k'\le n$,
$$
I_2(n,k,k')\le
\sum_{s=0}^n\left(\left(\sum_{j=1}^n\vert a_{j+s}\vert\right)^2\right)=
O\left(\sum_{s=1}^ns^{-1}\right)=O\left(\log(n)\right).
$$
So, we have
\begin{eqnarray*}\lefteqn{
\frac{1}{n^2\log^2 n}\sum_{k=1}^n\sum_{k'=1}^nI^2_2(n,k,k')=O\left(\frac{1}{n^2\log(n)}\sum_{k=1}^n\sum_{k'=1}^nI_2(n,k,k')\right)}\\
&&=\frac{O(1)}{n^2\log(n)}\sum_{k=1}^n\sum_{k'=1}^n\sum_{s=0}^n\\
&&\left(\left(\sum_{j=1}^k\frac{j}{n}\vert a_{j+s}\vert+\sum_{j=k+1}^n\left(1-\frac{j}{n}\right)\vert a_{j+s}\vert\right)\left(\sum_{j=1}^{k'}\frac{j}{n}\vert a_{j+s}\vert+\sum_{j=k'+1}^n\left(1-\frac{j}{n}\right)\vert a_{j+s}\vert\right)\right).
\end{eqnarray*}
The first product is equivalent to
\begin{eqnarray*}\lefteqn{
\frac{O(1)}{\log(n)}\sum_{k=1}^n\sum_{k'=1}^n\sum_{s=0}^n\sum_{j=1}^k\sum_{j'=1}^{k'}\left(\frac{j}{n}\left(\frac{j+s}{n}\right)^{-3/2}\frac{j'}{n}\left(\frac{j'+s}{n}\right)^{-3/2}\right)n^{-5}}\\
&&=\frac{O(1)}{\log(n)}\int_0^1\left(\int_0^1\left(\int_0^zx(x+y)^{-3/2}dx\right)^2dz\right)dy\to0\textrm{ as }n\to\infty,
\end{eqnarray*}
since the integral is finite. The other terms treat similarly. We now show that as $n\to\infty$,
\begin{equation}\label{third term}
\frac{1}{n^2\log^2 n}\sum_{k=1}^n\sum_{k'=1}^nI^2_3(n,k,k')\to0.
\end{equation}
\begin{eqnarray*}\lefteqn{
I^2_3(n,k,k')}\\
&&\le\left(\sum_{s=1}^n\left(\sum_{j=s}^n\left(\frac{j}{n}-\mathbf 1_{\{j>k\}}\right)a_{j-s}\right)^2\right)\left(\sum_{s=1}^n\left(\sum_{j=s}^n\left(\frac{j}{n}-\mathbf 1_{\{j>k'\}}\right)a_{j-s}\right)^2\right).
\end{eqnarray*}
Now, we have
\begin{eqnarray*}\lefteqn{
\left(\sum_{s=1}^n\left(\sum_{j=s}^n\left(\frac{j}{n}-\mathbf 1_{\{j>k\}}\right)a_{j-s}\right)^2\right)
}\\
&&\le2\sum_{s=1}^k\left[\left(\sum_{j=s}^k\frac{j}{n}a_{j-s}\right)^2+\left(\sum_{j=k+1}^n\left(\frac{j}{n}-1\right)a_{j-s}\right)^2\right]\\ 
&&+\sum_{s=k+1}^n\left(\sum_{j=s}^n\left(\frac{j}{n}-1\right)a_{j-s}\right)^2\\
&&\le4\sum_{s=1}^k\left(\frac{(k-s)}{n^2}+(k-s)^{-1}+(k-s)^{-1}\right) \\
&&+\sum_{s=k+1}^n\left(\frac{(n-s)^{1/2}}{n}+\left(1-\frac{s}{n}\right)(n-s)^{-1/2}\right)=O(\log(n)),
\end{eqnarray*}
uniformly in $k$ and hence $I_3(n,k,k')=O(\log(n))$ uniformly in $k,k'\le n$. Therefore, we can write
\begin{eqnarray}\label{a new star}\lefteqn{
\frac{1}{n^2\log^2 n}\sum_{k=1}^n\sum_{k'=1}^nI^2_3(n,k,k')=\frac{O(1)}{n^2\log(n)}\sum_{k=1}^n\sum_{k'=1}^nI_3(n,k,k')}\\
&&\le\frac{O(1)}{n^2\log(n)}\sum_{s=1}^n\sum_{k=1}^n\sum_{k'=1}^n\\
&&
\left(\left\vert\sum_{j=s}^k\frac{j}{n}a_{j-s}\right\vert+\left\vert\sum_{j=k+1}^n\left(\frac{j}{n}-1\right)a_{j-s}\right\vert\right)\left(\left\vert\sum_{j=s}^{k'}\frac{j}{n}a_{j-s}\right\vert+\left\vert\sum_{j=k'+1}^n\left(\frac{j}{n}-1\right)a_{j-s}\right\vert\right).
\end{eqnarray}
Decomposing the sum over $s$, we find that
\begin{eqnarray*}\lefteqn{
\frac{1}{n^2}\sum_{s=1}^n\sum_{k=s}^n\sum_{k'=s}^n\left(\left\vert\sum_{j=s}^k\frac{j}{n}a_{j-s}\right\vert\left\vert\sum_{j=s}^k\frac{j}{n}a_{j-s}\right\vert\right)}\\
&&=O\left(\frac{1}{n^2}\sum_{s=1}^n\sum_{k=s}^n\sum_{k'=s}^n\left(\frac{(k-s)^{1/2}+s(k-s)^{-1/2}}{n}\right)\left(\frac{(k'-s)^{1/2}+s(k'-s)^{-1/2}}{n}\right)\right).
\end{eqnarray*}
As $n\to\infty$, the normalized sum inside the parentheses converges to the integral
$$
\int_0^1\left(\int_y^1[(x-y)^{1/2}+y(x-y)^{-1/2}]dx\right)^2dy<\infty.
$$
\begin{eqnarray*}\lefteqn{
\frac{1}{n^2}\sum_{s=1}^n\sum_{k=s}^n\sum_{k'=s}^n\left(\sum_{j=k+1}^n\vert a_{j-s}\vert\sum_{j=k'+1}^n\vert a_{j-s}\vert\right)}\\
&&=O\left(\sum_{s=1}^n\sum_{k=s}^n\sum_{k'=s}^n\sum_{j=k+1}^n\sum_{j'=k'+1}^n\left(\frac{j-s}{n}\right)^{-3/2}\left(\frac{j'-s}{n}\right)^{-3/2}n^{-5}\right).
\end{eqnarray*}
As $n\to\infty$, the sum inside the parentheses converges to
$$\int_0^1\left(\int_y^1(x-y)^{-1/2}dx\right)^2dy=2.
$$
Also, one of the cross terms (the other treats identically) is bounded by
\begin{eqnarray*}\lefteqn{
\frac{1}{n^2}
\sum_{s=1}^n\sum_{k=s}^n\sum_{k'=s}^n\sum_{j'=k'+1}^n
\left\vert\sum_{j=s}^k\frac{j}{n}a_{j-s}\right\vert\sum_{j'=k'+1}^n\vert a_{j-s}\vert}\\
&&=O\left(\sum_{s=1}^n\sum_{k=s}^n\sum_{k'=s}^n\sum_{j'=k'+1}^n\left(\left(\frac{k-s}{n}\right)^{1/2}+\frac{s}{n}\left(\frac{k-s}{n}\right)^{-1/2}\right)\left(\frac{j'-s}{n}\right)^{-3/2}n^{-4}\right).
\end{eqnarray*}
The sum inside converges to the integral
$$
\int_0^1\left(\left(\int_y^1(z-y)^{1/2}+y(z-y)^{-1/2}dz\right)\int_y^1(x-y)^{-1/2}dx\right)dy<\infty.
$$
Summing over $s>\max(k,k')$ gives
\begin{eqnarray*}\lefteqn{
\frac{1}{n^2}\sum_{s=1}^n\sum_{k=1}^s\sum_{k'=1}^s\left(\left\vert\sum_{j'=s}^n\left(\frac{j}{n}-1\right)a_{j'-s}\right\vert\right)
\left(\left\vert\sum_{j=s}^n\left(\frac{j}{n}-1\right)a_{j-s}\right\vert\right)}\\
&&=O\left(\sum_{s=1}^n\left[\frac{(n-s)^{1/2}}{n}+\left(1-\frac{s}{n}\right)(n-s)^{-1/2}\right]^2\right)
\end{eqnarray*}
and the sum inside the parentheses converges to 
$$
\int_0^14xdx=2.
$$
Summing over $s$ between $\min(k,k')$ and $\max(k,k')$ gives the following (taking $k<k'$)
\begin{eqnarray*}\lefteqn{
\frac{1}{n^2}\sum_{k'=1}^n\sum_{k=1}^{k'}\sum_{s=k}^{k'}
\left(\left\vert\sum_{j=s}^n\left(\frac{j}{n}-1\right)a_{j-s}\right\vert\right)\left(\left\vert\sum_{j=s}^{k'}\frac{j}{n}a_{j-s}\right\vert+\left\vert\sum_{j=k'+1}^n\left(\frac{j}{n}-1\right)a_{j-s}\right\vert\right)}\\
&&=O\left(\frac{1}{n}\sum_{k'=1}^n\sum_{s=1}^{k'}n^{-1/2}\left(\frac{k'-s}{n}\right)^{-1/2}n^{-1/2}\right).
\end{eqnarray*}
The sum converges to 
$$
\int_0^1\int_0^x(x-y)^{-1/2}dydx=\frac{4}{3}.
$$
This completes the proof of (\ref{third term}), which in turn completes the proof of (\ref{lim var}).
\item We will just show that 
\begin{equation}\label{dif var}
\textrm{Var}(D_n)\sim\frac{\delta(-1/2)}{2}\log(n).
\end{equation}
The rest is quite similar to the proof of part A. With
$$
\bar X^{(m)}=\frac{1}{m}\sum_{i=n-m+1}^nX_i,
$$
write
\[
\mathrm{Var}(D_n)
=\mathrm{Var}(\bar X_m)+\mathrm{Var}(\bar X^{(m)})
-2\,\mathrm{Cov}(\bar X_m,\bar X^{(m)}).
\]
First, with $m=[\sqrt{n}],$ using (\ref{1/2 variance}), we obtain that
$$
\mathrm{Var}(\bar X_m)\sim\frac{\delta(-1/2)}{2}\log m\sim\frac{\delta(-1/2)}{4}\log(n).
$$
Using the fact that $n-m+i\sim n$, uniformly in $i\le m$, similar computations leading to (\ref{1/2 variance}) also show that
$$
\mathrm{Var}(\bar X^{(m)})\sim
\delta(-1/2)\log(n)-\frac{\delta(-1/2)}{2}\log m\sim
\frac{3\delta(-1/2)}{4}\log(n).
$$

Next,
\[
\mathrm{Cov}(\bar X_m,\bar X^{(m)})
=\frac{1}{m^2}\sum_{i=1}^m\sum_{j=1}^m \mathbb{E}(X_iX_{n-m+j}).
\]
Using
\[
\mathbb{E}(X_iX_{n-m+j})
=\frac12\Big(
\mathbb{E}(X_{n-m+j}^2)+\mathbb{E}(X_i^2)-\mathbb{E}(X_{n-m+j-i}^2)
\Big),
\]
and noting that both $\mathbb{E}(X_{n-m+j}^2)$ and $\mathbb{E}(X_{n-m+j-i}^2)$ are aymptotically equivalent  to $\delta(-1/2)\log(n)$ (uniformly in $i,j\le m$), these terms cancel, yielding
\[
\mathrm{Cov}(\bar X_m,\bar X^{(m)})
\sim \frac{1}{2m^2}\sum_{i,j=1}^m \mathbb{E}(X_i^2)
=\frac{1}{2m}\sum_{i=1}^m \mathbb{E}(X_i^2)
\sim \frac{\delta(-1/2)}{2}\log m
\sim \frac{\delta(-1/2)}{4}\log(n).
\]

Therefore,
\[
\mathrm{Var}(D_n)
\sim \frac{\delta(-1/2)}{4}\log(n)
+\frac{3\delta(-1/2)}{4}\log(n)
-2\cdot \frac{\delta(-1/2)}{4}\log(n)
=\frac{\delta(-1/2)}{2}\log(n).
\].
\item Reproducing computations similar to those in the proof of Theorem \ref{finite dim}, we obtain that for every fixed  $j\ge1$, and as $n\to\infty$,
\[
\Cov\!\left(
\frac{1}{n\sqrt{\log(n)}}\sum_{k=1}^n X_k,\;
\frac{1}{n}\sum_{k=1}^n \cos(k\lambda_j)\,X_k
\right)
=
O\!\left(\frac{1}{\sqrt{\log(n)}}\right)
\;\longrightarrow\;0.
\]
This yields the claimed asymptotic independence between $\bar X/\sqrt{\log(n)}$ and $Z_n(s,1/2)$. Similar argument holds for $D_n.$
\end{enumerate}
\end{proof}
\subsection{Proof of Proposition \ref{periodogram variance}}
\begin{proof}
{\bf Proof of Part 1.}
According to Theorem \ref{finite dim},
\begin{equation}\label{z zprime}
n^{-2d}\sum_{j=1}^sI_n(\lambda_j)=\|Z_n(s,d)\|^2\overset{d}{\longrightarrow}Z(d)'Z(d)
\end{equation}
where $Z(d)=\mathcal{N}\left(0,\Sigma(d)\right)$ where $\Sigma(d)$ is defined in (\ref{Sigma}). \\
{\bf Case  $\mathbf{d\in(-1/2,1/2)}$.}  Since $X_k$ is a stationary linear process and therefore ergodic, 
$$
D_n=\bar X_m-\frac{1}{m}\sum_{i=n-m+1}^nX_i\overset{P}{\to}\mathbb{E}(X_1)-\mathbb{E}(X_1)=0,
$$
and
$$\frac{1}{n}\sum_{k=1}^n(X_k-\bar X_n)^2\overset{P}{\to}\textrm{Var}(X_1)>0,
$$
so that
\begin{equation}\label{lim var1}
\frac{D_n^2}{\frac{1}{n}\displaystyle\sum_{k=1}^n(X_k-\bar X_n)^2}\overset{P}{\to}0.
\end{equation}
Moreover,
$$
\frac{\log (n)\displaystyle\sum_{j=1}^sI_n(\lambda_j)}{\displaystyle\sum_{k=1}^n(X_k-\bar X_n)^2}=n^{2d-1}\log(n)\frac{n^{-2d}\displaystyle\sum_{j=1}^sI_n(\lambda_j)}{n^{-1}\displaystyle\sum_{k=1}^n(X_k-\bar X)^2}\overset{P}{\to}0
$$
by \eqref{z zprime}, \eqref{lim var1} and Slutsky Theorem, and hence $\tilde Q_n(s)\overset{P}{\to}0$.\\
{\bf Case  $\mathbf{d=1/2}$.} From (\ref{Sigma}) and  (\ref{Sigma variance}), $\Sigma(1/2)=(1/2)\delta(-1/2)\Sigma$, the result is an immediate consequence of (\ref{z zprime}) and Theorem  \ref{1 over f1}.\\
{\bf Case $\mathbf{1/2<d<3/2}$.} 
\begin{equation}\label{final ratio}
\tilde Q_n(s)\ge\frac{\log (n)\displaystyle\sum_{j=1}^sI_n(\lambda_j)}{\displaystyle\sum_{k=1}^nX_k^2}=\frac{n^{-2d}\displaystyle\sum_{j=1}^sI_n(\lambda_j)}{\dfrac{1}{n^{2d}\log (n)}\displaystyle\sum_{k=1}^nX_k^2}.
\end{equation}
The numerator converges in distribution to a positive random variable by (\ref{z zprime}). Writing $X_k=\sum_{j=1}^kY_j+X_0$ where $Y_j$ is a zero-mean linear stationary process with memory parameter $d-1$, we obtain that
$$
\mathbb{E}(X_k^2)=O\left(k^{2(d-1)+1}\right),\qquad\textrm{as }k\to\infty,
$$
we obtain that the denominator is $O_P(1/\log(n))=o_P(1)$, as $n\to\infty$, so that the right hand ratio in (\ref{final ratio}) converges  in probability to infinity.\\\\
{\bf Proof of Part 2.}
Writing $X_t=g_n(t)+Y_t=n^\beta g\left(\frac{t}{n}\right)+Y_t$, we have
\begin{eqnarray}\label{decomp}
I_n(\lambda_j)&=&\vert D_{n,Y}(j)\vert^2+n^{2\beta+1}\vert D_{n,g}(j)\vert^2
+2n^{\beta+1/2}\,\,\textrm{Re}\left(D_{n,g}(j)\overline{D_{n,Y}(j)}\right)
\end{eqnarray}
where
$$
D_{n,g}(j)=\sum_{t=1}^ng\left(\frac{t}{n}\right)e^{\i 2\pi t j/n}\frac{1}{n},
$$
and
$$
D_{n,Y}(j)=\frac{1}{\sqrt{ n}}\sum_{t=1}^nY_te^{\i 2\pi t j/n}.
$$
Under assumption (\ref{nonzero}), for at least one index $j=1,\ldots,s$, we obtain via  Riemann sum approximation of an integral,
$$
\vert D_{n,g}(j)\vert^2\sim c>0, \textrm{ as }n\to\infty,\qquad\textrm{for some positive constant }c.
$$
Since $Y_t$ is a zero-mean linear stationary process with memory parameter $d\in(-1/2,1/2)$, then from Theorem \ref{finite dim}, $D_{n,Y}(j)=O_P(n^d)$ we have as $n\to\infty$, and hence, since $\beta\ge0$,
\begin{equation}\label{num1}
\frac{I_n(\lambda_j)}{ n^{2\beta+1}}\overset{P}{\to}c.
\end{equation}
Also, since $Y_k$ is a zero mean stationary process and using the fact that $g$ is bounded,
\begin{eqnarray*}\lefteqn{
\mathbb{E}\left(\frac{1}{n^{2\beta+1}\log(n)}\sum_{k=1}^n(X_k-\bar X)^2\right)}\\
&&\le\frac{1}{n^{2\beta+1}\log(n)}\sum_{k=1}^n\mathbb{E}(X_k^2)=\frac{1}{n^{2\beta+1}\log(n)}\sum_{k=1}^n\left(n^{2\beta}\left(g\left(\frac{k}{n}\right)\right)^2+\mathbb{E}(Y_1^2)\right)\\ 
&&=O\left(\frac{1}{\log(n)}\right)\to0,\qquad\textrm{as }n\to\infty,
\end{eqnarray*}
and therefore
$$
\frac{\log(n)I_n(\lambda_j)}{\sum_{k=1}^n(X_k-\bar X)^2}\overset{P}{\to}\infty
$$
and hence $\tilde Q_n(s)\overset{P}{\to}\infty.$
\end{proof}
\bibliographystyle{apalike}

\end{document}